\documentclass{amsart}
\usepackage{latexsym,amssymb}
\usepackage[english]{babel}
\usepackage{bm}

\usepackage{amsfonts, amsthm}

\newtheorem{prop}{Proposition}[section]

\newtheorem{cor}[prop]{Corollary}
\newtheorem{thm}[prop]{Theorem}

\newtheorem{conj}{Conjecture}
\theoremstyle{definition}

\newtheorem{rem}[prop]{Remark}

\newtheorem{defi}[prop]{Definition}
\newtheorem{ex}[prop]{Example}

\def\S{\mathbb{S}}
\def\Z{\mathbb{Z}}

\newcommand{\nn}[1]{\left\|{#1}\right\|}
\newcommand{\gr}[1]{\bm{#1}}

\def\G{\Gamma}
\def\SH{\mathrm{SH}}
\def\H{\mathrm{H}}

\def\S{\mathcal{S}}

\def\q{\square}

\begin{document}

\title[Globally simple Heffter arrays and \ldots]{Globally simple Heffter arrays and orthogonal cyclic cycle decompositions}

\author{S. Costa}
\address{DICATAM - Sez. Matematica, Universit\`a degli Studi di Brescia, Via
Branze 43, I-25123 Brescia, Italy}
\email{simone.costa@unibs.it}
\thanks{The results of this paper have been presented at HyGraDe 2017}

\author{F. Morini}
\address{Dipartimento di Scienze Matematiche, Fisiche e Informatiche, Universit\`a di Parma,
Parco Area delle Scienze 53/A, I-43124 Parma, Italy}
\email{fiorenza.morini@unipr.it}

\author{A. Pasotti}
\address{DICATAM - Sez. Matematica, Universit\`a degli Studi di Brescia, Via
Branze 43, I-25123 Brescia, Italy}
\email{anita.pasotti@unibs.it}

\author{M.A. Pellegrini}
\address{Dipartimento di Matematica e Fisica, Universit\`a Cattolica del Sacro Cuore, Via Musei 41,
I-25121 Brescia, Italy}
\email{marcoantonio.pellegrini@unicatt.it}

\begin{abstract}
In this paper we introduce a particular class of Heffter arrays, called globally simple Heffter arrays, whose existence
gives at once orthogonal cyclic cycle decompositions of the
complete graph and of the cocktail party graph. In particular we provide explicit
constructions  of such decompositions for cycles of length $k\leq 10$.
Furthermore, starting from our Heffter arrays we also obtain biembeddings of two $k$-cycle decompositions on orientable surfaces.
\end{abstract}


\maketitle
\section{Introduction}

Arrays with particular properties  are  not only interesting objects \emph{per se} but, in general, they have applications in many areas of mathematics.
For these reasons, there are several types of well-studied arrays, see, for instance, \cite{DK,DS, H1,H2,HSS,KSW}.
Here we consider Heffter arrays, introduced by  Archdeacon in \cite{A}:

\begin{defi}\label{def:H}
An \emph{integer Heffter array} $\H(m,n; h,k)$ is an $m \times n$ partially filled array such that:
\begin{itemize}
\item[(\rm{a})] its entries belong to the set $\{\pm 1, \pm 2, \ldots, \pm nk\}\subset \Z$;
\item[(\rm{b})] no two entries agree in absolute value;
\item[(\rm{c})] each row contains $h$ filled cells and each column contains $k$ filled cells;
\item[(\rm{d})] the elements in every row and column sum to $0$.
\end{itemize}
\end{defi}

Trivial necessary conditions for the existence of an $\H(m,n; h,k)$ are $mh=nk$, $3\leq h\leq n$ and $3\leq k
\leq m$.
In this paper we will concentrate on \emph{square} integer Heffter arrays, namely on the case $m=n$ which implies $h=k$.
An $\H(n,n;k,k)$ will be simply denoted by $\H(n;k)$.

\begin{ex}\label{ex:H87}
The following array is an example of an $\H(8;7)$:
\begin{center}
\begin{footnotesize}
$\begin{array}{|r|r|r|r|r|r|r|r|}\hline
8 & 16 &   & 25 & -27 & -29 & 31 & -24 \\
\hline  -17 & -6 & 23 & -28 & 26 & 32 & -30 &   \\
\hline  39 & -10 & -5 & 15 &   & 33 & -35 & -37 \\
\hline  -38 &   & -18 & 7 & 11 & -36 & 34 & 40 \\
\hline  -43 & -45 & 47 & -22 & 3 & 19 &   & 41 \\
\hline  42 & 48 & -46 &   & -14 & 2 & 12 & -44 \\
\hline    & 49 & -51 & -53 & 55 & -21 & 1 & 20 \\
\hline  9 & -52 & 50 & 56 & -54 &   & -13 & 4  \\\hline
\end{array}$
\end{footnotesize}
\end{center}
\end{ex}
The existence problem of \emph{square integer} Heffter arrays has been completely solved in \cite{ADDY,DW}, where the authors
proved the following theorem.

\begin{thm}\label{th:integer}
There exists an $\H(n;k)$ if and only if
$$3\leq k\leq n\quad \textrm{and}\quad nk\equiv 0,3\pmod{4}.$$
\end{thm}

If, in Definition \ref{def:H}, condition (d) is replaced by the following one:
\begin{itemize}
\item[(\rm{d}$'$)] the elements in every row and column sum to $0$ modulo $2nk+1$,
\end{itemize}
one speaks of a non integer Heffter array, see \cite{A}. The existence problem of \emph{non integer} Heffter arrays has not been completely solved
yet:  some partial results can be found in \cite{ABD,DM}.

In \cite{A}, Archdeacon showed that such arrays can be used to construct cycle decompositions of the complete graph 
if they satisfy an additional condition, called simplicity, which we introduce next.
Let $A$ be a finite subset of $\Z\setminus\{0\}$. Given an ordering $\omega=(a_1,a_2,\ldots,a_k)$ of the elements in $A$, let $s_i=\sum_{j=1}^i a_j$
be the $i$-th partial sum of $A$.
We say that the ordering $\omega$ is \emph{simple modulo $v$} if $s_b\neq s_c\pmod{v}$ for all $1\leq b <  c\leq k$ or,
equivalently,
if there is no proper subsequence of $\omega$ that sums to $0$ modulo $v$.

With a little abuse of notation, we will identify each row (column) of an $\H(n;k)$ with the set of size $k$
whose elements are the entries of the nonempty cells of such a row (column). For instance we can view the first
row of the $\H(8;7)$ of Example \ref{ex:H87} as the set $R_1=\{8,16,25,-27,-29,31,-24\}$.

\begin{defi}
An $\H(n;k)$ is said to be \emph{simple} if each row and each column admits a simple ordering modulo $2nk+1$.
\end{defi}

Since each row and each column of an $\H(n;k)$ is such that $s_k=0$ and does not contain $0$ or subsets of the form
$\{x,-x\}$,
it is easy to see that every $\H(n;k)$ with $k\leq 5$ is simple.

\begin{ex}\label{ex:simple}
The $\H(8;7)$ of Example \ref{ex:H87} is simple. To verify this property we need to provide an ordering for each row and each column which is simple modulo $113$.
One can check that the $\omega_i$'s are simple ordering of the rows and the  $\nu_i$'s are simple ordering of the columns:

\begin{center}
\begin{footnotesize}
$\begin{array}{lcl}
\omega_1= ( 8, 25, 16, -27, -29, 31, -24); & \quad &\nu_1= (8,39,-17,-38,-43,42,9  );\\
\omega_2= (-17, -6, -28, 23, 26, 32, -30); & \quad &\nu_2= ( 16,-6,-45,-10,48,49,-52 );\\
\omega_3= (39, -10, -5, 33, 15, -35, -37); & \quad & \nu_3= ( 23,-5,47,-18,-46,-51,50 );\\
\omega_4= (-38, -18, 7, -36, 11, 34,  40); & \quad &\nu_4= (25,-28,15,7,-53,-22,56 );\\
\omega_5= (-43, -45, 47, -22, 3, 41,  19); & \quad &\nu_5= (-27,26,11,3,55,-14,-54 );\\
\omega_6= ( 42, 48, -46, -14, 2, -44, 12); & \quad &\nu_6= ( -21,32,33,-36,19,2,-29);\\
\omega_7= (20, -51, -53, 55, -21, 1, 49);& \quad  &\nu_7= (-13,-30,-35,34,12,1,31);\\
\omega_8= (-52, 9, 50, 56, -54, -13, 4); & \quad &\nu_8= (-37,-24,40,41,-44,20,4).
\end{array}$
\end{footnotesize}
\end{center}
\end{ex}

Notice that the \emph{natural ordering} of each row (from left to right)  and of each column  (from top to bottom) of the above $\H(8;7)$ is not
simple.
Clearly, larger $n$ and $k$ are more difficult (and tedious) is to provide
explicit simple orderings for rows and columns of an $\H(n;k)$.
Hence, we think it is reasonable to look for Heffter arrays which are simple with respect to the natural ordering
of rows and columns.
So we propose the following new definition.

\begin{defi}
We say that an Heffter array $\H(n; k)$ is \emph{globally simple} if each row and each column is simple
modulo $2nk+1$ with respect to their natural ordering.
\end{defi}

Clearly, the concepts of simple and globally simple Heffter array can be extended to the rectangular case
$\H(m,n; h,k)$.
A globally simple $\H(m,n; h,k)$ will be denoted by $\SH(m,n; h,k)$ and a square globally simple $\H(n;k)$ will be
denoted by $\SH(n;k)$.

\begin{ex}
The following is a globally simple $\SH(8;7)$:
\begin{center}
\begin{footnotesize}
\begin{tabular}{|r|r|r|r|r|r|r|r|}
\hline $4$ & $35$ & $-45$ & $46$ & $ $ & $20$ & $-36$ & $-24$ \\
\hline $48$ & $-5$ & $23$ & $-47$ & $-18$ & $ $ & $37$ & $-38$ \\
\hline $-32$ & $-10$ & $-6$ & $31$ & $-41$ & $42$ & $ $ & $16$ \\
\hline $33$ & $-34$ & $44$ & $3$ & $11$ & $-43$ & $-14$ & $ $ \\
\hline $ $ & $15$ & $-28$ & $-22$ & $7$ & $27$ & $-53$ & $54$ \\
\hline $-13$ & $ $ & $29$ & $-30$ & $56$ & $1$ & $12$ & $-55$ \\
\hline $-49$ & $50$ & $ $ & $19$ & $-40$ & $-21$ & $2$ & $39$ \\
\hline $9$ & $-51$ & $-17$ & $ $ & $25$ & $-26$ & $52$ & $8$ \\\hline
\end{tabular}
\end{footnotesize}
\end{center}
\end{ex}

\begin{rem}
We point out that the Heffter arrays constructed in \cite{ADDY, DW}, in general, are not globally simple. Indeed, the $\H(8;7)$ of Example \ref{ex:H87} was obtained according to \cite[Theorem 3.12]{ADDY}.
Furthermore, easy modifications of the existing constructions seem not to produce globally simple arrays, except for $\SH(n;6)$: these arrays (see Proposition \ref{prop:6}) are obtained switching the first two columns of the matrices given in \cite[Theorem 2.1]{ADDY}.
\end{rem}

In \cite{A} many applications of (simple) Heffter arrays are shown, in particular the relationship with orthogonal cycle decompositions of the complete graph
and with biembeddings of two cycle decompositions on an orientable surface.
Here, in Section \ref{sec2} we show that globally simple Heffter arrays are related not only to orthogonal cyclic cycle decompositions
of the complete graph, but also of the cocktail party graph.
We note that very little is known about orthogonal decompositions; as far as we know, only
asymptotic results have been obtained, see \cite{CY1, CY2}.
In Section \ref{sec3} we investigate the connection between Heffter arrays and biembeddings of two cycle decompositions on an orientable surface.
Then in Section \ref{sec4} we present direct constructions of $\SH(n;k)$ for $6\leq k\leq 10$ and for any admissible $n$.
Combining the results of Sections \ref{sec2} and \ref{sec4} we obtain the following theorem.

\begin{thm}\label{th:main}
Let $3\leq k\leq 10$.
Then there exists a pair of orthogonal cyclic $k$-cycle decompositions of the complete graph of order $2nk+1$ and of the cocktail party graph of order $2nk+2$ for any positive integer $n$  such that  $nk \equiv 0,3 \pmod 4$.
\end{thm}

We have to point out that for $3\leq k\leq 9$, Theorem \ref{th:main} can be obtained starting from the results 
of \cite{ADDY, CMPP, DW}.
But, in that case, if one wants to construct the base cycles for the cycle decompositions of order $2nk+1$ he has to find an ad hoc simple ordering for each row and each column,
then he has to find other simple orderings modulo $2nk+2$.
While here the cycle decompositions (both of the complete graph of order $2nk+1$ and of the cocktail party graph of order $2nk+2$) can be immediately written  starting from the rows and columns of the arrays constructed in Section \ref{sec4}.

It is worth noticing that combining the previous theorem with \cite[Theorem 3.3]{BGL}, a stronger result can be stated regarding cocktail party graphs.

\begin{cor}
Let $3\leq k\leq 10$ and $n\geq 1$.
Then there exists a pair of orthogonal cyclic $k$-cycle decompositions of the cocktail party graph of order $2nk+2$ if and only if  $nk \equiv 0,3 \pmod 4$.
\end{cor}

Finally, combining the results of Sections \ref{sec3} and \ref{sec4} we obtain the following theorem.

\begin{thm}\label{thm:biemb}
There exists a biembedding of the complete graph of order $2nk+1$ and one of the cocktail graph of order $2nk+2$ on 
orientable surfaces such that every face is a $k$-cycle, whenever $k\in \{3,5,7,9\}$, $nk\equiv 3 \pmod 4$ and $n>k$.
\end{thm}

\section{Orthogonal cyclic cycle decompositions}\label{sec2}

We first recall some basic definitions about graph decompositions.
Let $\G$ be a graph with $v$ vertices. A $k$-\emph{cycle decomposition} of $\G$ is a set $\mathcal C$ of $k$-cycles
of $\G$
such that each edge of $\G$ belongs to a unique cycle of  $\mathcal C$.
If $\G$ is the complete graph of order $v$, one also speaks of a $k$-cycle system of order $v$.
A  $k$-cycle decomposition of $\G$ is said to be \emph{cyclic} if it admits $\Z_{v}$
as automorphism group acting sharply transitively on the vertices.
We recall the following result.

\begin{prop}
Let $\G$ be a graph with $v$ vertices. A $k$-cycle decomposition $\mathcal{C}$ of $\G$ is sharply vertex-transitive under $\Z_v$ if and only if,
up to isomorphisms, the following conditions hold:
\begin{itemize}
\item the set of vertices of $\G$ is $\Z_v$;
\item for all $C=(c_1,c_2,\ldots,c_{k})\in \mathcal{C}$,  $C+1:=(c_1+1,c_2+1,\ldots,c_{k}+1)\in \mathcal{C}$.
\end{itemize}
\end{prop}

Clearly, to describe a cyclic  $k$-cycle decomposition it is sufficient to exhibit a complete system $\mathcal{B}$
of representatives for the orbits of $\mathcal{C}$ under the action of $\Z_v$. The elements of $\mathcal{B}$
are called \emph{base cycles} of $\mathcal{C}$.

Here we are interested in the cases in which $\Gamma$ is either the complete graph $K_v$ whose vertex-set is $\Z_v$ or
the cocktail party graph $K_{2t}-I$, namely the complete graph $K_{2t}$ minus the $1$-factor $I$ whose edges are $[0,t],[1,t+1],[2,t+2],\ldots,[t-1,2t-1]$.
The problem of finding necessary and sufficient conditions for cyclic $k$-cycle decompositions of $K_v$ and $K_{2t}-I$ has attracted much attention (see, for instance, \cite{BDF,V,WF} and \cite{BGL,BR,JM2008,JM2017}, respectively).
One of most efficient tools applied for solving this problem is the \emph{difference method}.

\begin{defi}
Let  $C=(c_1,c_2,\ldots,c_{k})$ be a $k$-cycle with
vertices in $\Z_v$.
The multiset
$$\Delta C = \{\pm(c_{h+1}-c_{h})\mid  1\leq h \leq k\},$$
where   the subscripts are taken modulo $k$,
is called the \emph{list of differences} from $C$.
\end{defi}

More generally, given a set $\mathcal{B}$ of $k$-cycles with vertices
in $\Z_v$, by $\Delta \mathcal{B}$ one means the union (counting
multiplicities) of all multisets $\Delta C$,
where  $C\in \mathcal{B}$.

\begin{thm}\label{thm:basecycles}
Let $\mathcal{B}$ be a set of $k$-cycles with vertices in $\Z_{v}$.
\begin{itemize}
\item[(1)] If $\Delta\mathcal{B}=\Z_v\setminus\{0\}$ then $\mathcal{B}$ is a set of base cycles of
a cyclic $k$-cycle decomposition of $K_v$.
\item[(2)] If $v=2t$ and $\Delta\mathcal{B}=\Z_{2t}\setminus\{0,t\}$ then $\mathcal{B}$ is a set of base cycles of
a cyclic $k$-cycle decomposition of $K_{2t}-I$.
\end{itemize}
\end{thm}

Here, we are interested in constructing pairs of orthogonal $k$-cycle decompositions according to the following definition.

\begin{defi}
Two $k$-cycle decompositions $\mathcal{C}$ and $\mathcal{C}'$ of a graph $\G$ are said to be \emph{orthogonal} if for any cycle $C\in \mathcal{C}$
and  any cycle $C'\in \mathcal{C}'$, $C$ intersects $C'$ in at most one edge.
\end{defi}

\noindent Clearly, the same definition can be given for two arbitrary graph decompositions, see \cite{AHL}.

Starting from a simple $\H(n;k)$ it is possible to construct two orthogonal cyclic
$k$-cycle  decompositions of $K_{2nk+1}$, see \cite[Proposition 2.1]{A}.
Firstly, we have to find a simple ordering modulo $2nk+1$ for each row and each column. Then starting from
the simple orderings
of the rows we can construct a set $\mathcal{B}$ of base cycles of a cyclic $k$-cycle decomposition $\mathcal{C}$ of  $K_{2nk+1}$.
The vertices of the $i$-th cycle of $\mathcal{B}$ are the partial sums modulo $2nk+1$ of the $i$-th row of $\H(n;k)$.
Analogously, we can obtain a set of base cycles $\mathcal{B}'$ of another cyclic $k$-cycle decomposition $\mathcal{C}'$ of  $K_{2nk+1}$
starting from the simple orderings of the columns. The decompositions $\mathcal C$ and $\mathcal{C}'$ are orthogonal.

\begin{ex}
Let $H$ be the $\H(8;7)$ of Example \ref{ex:H87} and consider the simple orderings $\omega_i$'s and $\nu_i$'s given in
Example \ref{ex:simple}.
By the partial sums of the $\omega_i$'s ($\nu_i$'s, respectively) in $\Z_{113}$ we obtain the cycles $C_i$'s ($C'_i$'s, respectively):
\begin{center}
\begin{footnotesize}
$\begin{array}{lcl}
C_1= ( 8, 33, 49, 22, -7, 24, 0); & \quad &C'_1= (8,47,30,-8,-51,-9,0);\\
C_2= (-17, -23, -51, -28, -2, 30, 0); & \quad &C'_2= ( 16,10,-35,-45,3,52,0 );\\
C_3= (39, 29, 24, 57, 72, 37, 0); & \quad &C'_3= ( 23,18,65,47,1,-50,0);\\
C_4= (-38, -56, -49, -85, -74, -40, 0); & \quad &C'_4= (25,-3,12,19,-34,-56,0);\\
C_5= (-43, -88, -41, -63, -60, -19,  0); & \quad &C'_5= (-27,-1,10,13,68,54,0);\\
C_6= ( 42, 90, 44, 30, 32, -12, 0); & \quad &C'_6= ( -21,11,44,8,27,29,0);\\
C_7= (20, -31, -84, -29, -50, -49, 0);& \quad  &C'_7= (-13,-43,-78,-44,-32,-31,0);\\
C_8= (-52, -43, 7, 63, 9, -4, 0); & \quad &C'_8= (-37,-61,-21,20,-24,-4,0).
\end{array}$
\end{footnotesize}
\end{center}
Then $\mathcal{B}=\{C_1,\ldots,C_8\}$ and $\mathcal{B}'=\{C'_1,\ldots,C'_8\}$ are two sets of base cycles
of a pair of orthogonal cyclic $7$-cycle decompositions of  $K_{113}$.
\end{ex}

Although the existence of a square integer $\H(n;k)$ has been completely established, the simplicity of these arrays has
not been considered. In \cite{CMPP}, we proposed the following conjecture whose validity would imply that
any Heffter array is simple (other related and interesting conjectures can be found in \cite{ADMS}).

\begin{conj}
Let $(G,+)$ be an abelian group. Let $A$ be a finite subset of $G\setminus\{0\}$
such that no $2$-subset $\{x,-x\}$ is contained in $A$ and
with the property that $\sum_{a\in A} a=0$.
Then there exists a simple ordering of the elements of $A$.
\end{conj}

We proved that our conjecture is true for any subset $A$ of size less than $10$.
Our proof is constructive, but given an $\H(n;k)$ it can be long and tedious to find the required $2n$
simple orderings.
This is why we came up with the idea of introducing globally simple Heffter arrays.
Moreover, we will construct globally simple integer Heffter arrays $\SH(n;k)$ which satisfy also the following condition:
\begin{center}
$(\ast)$\quad the natural ordering of each row and column is simple modulo $2nk+2$.
\end{center}
The usefulness of these arrays, which will be denoted by $\SH^{*}(n;k)$, is explained by the following proposition.

\begin{prop}\label{pr:ccp}
If there exists an $\SH^*(n;k)$, then there exist:
\begin{itemize}
\item[(1)] a pair of orthogonal cyclic $k$-cycle decompositions of $K_{2nk+1}$ and
\item[(2)] a  pair of orthogonal cyclic $k$-cycle decompositions of $K_{2nk+2}-I$.
\end{itemize}
\end{prop}

\begin{proof}
(1) follows from previous considerations. (2)  As the natural ordering of each row is simple modulo $2nk+2$,
the partial sums of each row in $\Z_{2nk+2}$  are the vertices of a $k$-cycle. Let $\mathcal{B}$ be the set of the $k$-cycles so constructed from the rows.
Since $\Delta \mathcal{B}=\Z_{2nk+2}\setminus \{ 0,nk+1\}$, in view of Theorem \ref{thm:basecycles}, $\mathcal{B}$ is a set of base cycles of a cyclic $k$-cycle decomposition $\mathcal{C}$ of $K_{2nk+2}-I$. Analogously, starting from the columns,  we can obtain  another cyclic $k$-cycle decomposition $\mathcal{C}'$ of $K_{2nk+2}-I$.
By construction, the decompositions $\mathcal{C}$ and $\mathcal{C}'$ are orthogonal.
\end{proof}

\begin{ex}\label{10.8}
The following is an $\SH^*(10;8)$:
\begin{center}
\begin{footnotesize}
$\begin{array}{|r|r|r|r|r|r|r|r|r|r|}\hline
77 & 80 & -78 & -71 & -70 & -79 &  &  & 69 & 72 \\\hline
 &  & -17 & -20 & -25 & -28 & 26 & 19 & 18 & 27 \\\hline
5 & 8 & 13 & 16 & -14 & -7 & -6 & -15 &  &  \\\hline
34 & 43 &  &  & -33 & -36 & -41 & -44 & 42 & 35 \\\hline
 &  & 21 & 24 & 29 & 32 & -30 & -23 & -22 & -31 \\\hline
58 & 51 & 50 & 59 &  &  & -49 & -52 & -57 & -60 \\\hline
-38 & -47 &  &  & 37 & 40 & 45 & 48 & -46 & -39 \\\hline
-73 & -76 & 74 & 67 & 66 & 75 &  &  & -65 & -68 \\\hline
-62 & -55 & -54 & -63 &  &  & 53 & 56 & 61 & 64 \\\hline
-1 & -4 & -9 & -12 & 10 & 3 & 2 & 11 &  &  \\\hline
\end{array}$
\end{footnotesize}
\end{center}
By the partial sums in $\Z_{162}$ of the natural simple orderings of the rows (columns, respectively)  we obtain the cycles
$C_i$'s ($C'_i$'s, respectively):
\begin{center}
\begin{footnotesize}
$\begin{array}{lcl}
C_1= ( 77,   157,    79,     8,   -62,  -141,   -72,     0 ); & \;\; &C'_1= (77,    82,   116,   12,   136,    63,     1,     0);\\
C_2= (-17,   -37,   -62,   -90,   -64,   -45,   -27,     0 ); & \;\; &C'_2= (80,    88,   131,   20,   135,    59,     4,     0  );\\
C_3= (5,    13,    26,    42,    28,    21,    15,     0 ); & \;\; &C'_3= ( -78,   -95,   -82,   -61,   -11,    63,     9,     0 );\\
C_4= (34,    77,    44,     8,   -33,   -77,   -35,     0); & \;\; &C'_4= (-71,   -91,   -75,   -51,     8,    75,    12,     0);\\
C_5= (21,    45,    74,   106,    76,    53,    31,0); & \;\;&C'_5= (-70,   -95,  -109,  -142,  -113,   -76,   \\
& & -10,     0 );\\
C_6= (58,   109,   159,   56,   7,   117,    60,     0); & \;\; &C'_6= (-79,  -107,  -114,  -150,  -118,   -78,\\
&&   -3,     0  );\\
C_7= (-38,   -85,   -48,    -8,    37,    85,    39,     0 );& \;\;  &C'_7= ( 26,    20,   -21,   -51,  -100,   -55,    -2,     0 );\\
C_8= ( -73,  -149,   -75,    -8,    58,   133,    68,     0 ); & \;\; &C'_8= (19,     4,   -40,   -63,  -115,   -67,   -11,     0 );\\
C_9=(-62,  -117,  -9,  -72,  -19,  -125,   -64,     0 ); &\;\;&C'_9=( 69,    87,   129,   107,    50,     4,   -61,     0 );\\
C_{10}=(-1,    -5,   -14,   -26,   -16,   -13,   -11,     0); & \;\; &C'_{10}=(72,    99,   134,   103,    43,     4,   -64,     0 ).
\end{array}$
\end{footnotesize}
\end{center}
Then $\mathcal{B}=\{C_1,\ldots,C_{10}\}$ and $\mathcal{B}'=\{C'_1,\ldots,C'_{10}\}$ are two sets of base cycles
of a pair of orthogonal cyclic $8$-cycle decompositions of  $K_{162}-I$.\\
Analogously, if we consider the partial sums of each row (column, respectively) in $\Z_{161}$, we obtain the cycles
$\widetilde{C}_i$'s ($\widetilde{C}'_j$'s, respectively):
\begin{center}
\begin{footnotesize}
$\begin{array}{lcl}
\widetilde{C}_i= C_i, \quad i\ne 6,9;  &\quad&  \widetilde{C}'_j=C_j, \quad j\ne 1,2;\\
\widetilde{C}_6= (58,   109,   159,   57,   8,   117,    60,     0);& \quad &\widetilde{C}'_1= (77,    82,   116,   13,   136,    63,     1,     0);\\
\widetilde{C}_9=(-62,  -117,  -10,  -73,  -20,  -125,   -64,     0 ); & \quad &\widetilde{C}'_2= (80,    88,   131,   21,   135,    59,     4,     0  ).
\end{array}$
\end{footnotesize}
\end{center}
 Now $\widetilde{\mathcal{B}}=\{\widetilde{C}_1,\ldots,\widetilde{C}_{10}\}$ and $\widetilde{\mathcal{B}}'=\{\widetilde{C}'_1,\ldots,\widetilde{C}'_{10}\}$ are two sets of base cycles
of a pair of orthogonal cyclic $8$-cycle decompositions of  $K_{161}$.
\end{ex}

\section{Biembeddings of cycle decompositions}\label{sec3}

This section is dedicated to  the connection between Heffter arrays and biembeddings of two cycle decompositions on an orientable surface.
We recall that an embedding of a graph with each edge on a face of size $k$ and a face of size $h$ is called a \emph{biembedding}.
Notice that a biembedding is $2$-colorable with the faces that are $k$-cycles receiving 
one color while those faces that are $h$-cycles receiving the other color (see \cite{DM}).

Consider now a generic partially filled square array $A$ of size $n$ such that its $N$ nonempty
entries are pairwise distinct.
As usual,  we identify each row (column) of $A$ with the set whose elements are the entries of the nonempty cells of
such a row (column).
Let $\omega_r$ ($\omega_c$, respectively) be any ordering of the rows (columns, respectively) of $A$.
We say that $\omega_r$ and $\omega_c$ are two \emph{compatible} orderings if $\omega_r\circ \omega_c$ is a cycle of order $N$.

In particular, the following result holds:

\begin{thm}\label{Gustin}
Given a Heffter array $H=\H(m,n;h,k)$ with simple compatible orderings modulo $2nk+1$ ($2nk+2$, respectively) $\omega_r$ on the rows and $\omega_c$ on the columns of $H$, there exists a biembedding of $K_{2nk+1}$ (of $K_{2nk+2}- I$, respectively) on an orientable surface such that every edge is on a simple cycle face of size $k$ and a simple cycle face of size $h$.  
\end{thm}

\begin{proof}
The result for complete graphs was obtained by Archdeacon \cite[Sections 3 and 4]{A} using current graphs. This construction is based on a paper of Gustin \cite{G} that works in general for Cayley graphs, implying the result for cocktail graphs.
\end{proof}

Let $k\geq 1$ be an odd integer and let $A=(a_{i,j})$ be a partially filled square array of size $n$. We say that the
element $a_{i,j}$ belongs to the
diagonal $D_{s}$ if $j-i\equiv s-1 \pmod{n}$.
Moreover, $A$ is said to be cyclically $k$-diagonal if
the nonempty cells of $A$ are exactly those of the diagonals $D_s$ with $s\in \{r,\dots,r+k-1\}$ for a suitable integer
$r\in \{1,\dots,n\}$.

\begin{ex}\label{ex1}
The following partially filled array $A$ of size $n=9$  is cyclically $5$-diagonal  (with $r=8$):
\begin{center}
\begin{footnotesize}
$\begin{array}{|r|r|r|r|r|r|r|r|r|}\hline
38 &  39 &  40 &    &   &    &    &  36 &  37 \\ \hline
    42 &  43 &  44 &  45 &    &    &    &    &  41 \\ \hline
     1 &   2 &   3 &   4 &   5 &    &    &    &    \\ \hline
      &   6 &   7 &   8 &   9 &  10 &    &    &    \\ \hline
      &    &  11 &  12 &  13 &  14 &  15 &    &    \\ \hline
      &    &    &  16 &  17 &  18 &  19 &  20 &    \\ \hline
      &    &    &    &  21 &  22 &  23 &  24 &  25 \\ \hline
    30 &    &    &    &    &  26 &  27 &  28 &  29 \\ \hline
    34 &  35 &    &    &    &    &  31 &  32 &  33 \\ \hline
    \end{array}$
   \end{footnotesize}
   \end{center}
  \end{ex}

Given a cyclically $k$-diagonal array $A$ whose nonempty cells belong to the diagonals $D_r,D_{r+1},\dots, D_{r+k-1}$,
we can relabel its elements setting $b_{i,j}=a_{i-r+1,j}$, where the indices are considered modulo $n$ in such a way
that they belong to the set $\{1,\ldots,n\}$. We obtain a
partially filled array $B$ of size $n$ which is still cyclically $k$-diagonal but with  nonempty diagonals $D_1,\dots,D_k$. We call
such $B$ the standard form of $A$.

Note that this procedure has no influence on any orderings $\omega_r$ and $\omega_c$ of the rows and of the columns of
$A$, respectively.

\begin{ex}\label{ex2}
Starting from the array $A$ of Example \ref{ex1} we get the following $B$:
   \begin{center}
\begin{footnotesize}
$\begin{array}{|r|r|r|r|r|r|r|r|r|}\hline
     1 &   2 &   3 &   4 &   5 &    &    &    &    \\\hline
      &   6 &   7 &   8 &   9 &  10 &    &    &    \\\hline
      &    &  11 &  12 &  13 &  14 &  15 &    &    \\\hline
      &    &    &  16 &  17 &  18 &  19 &  20 &    \\\hline
      &    &    &    &  21 &  22 &  23 &  24 &  25 \\\hline
    30 &    &    &    &    &  26 &  27 &  28 &  29 \\\hline
    34 &  35 &    &    &    &    &  31 &  32 &  33 \\\hline
    38 &  39 &  40 &    &    &    &    &  36 &  37 \\\hline
    42 &  43 &  44 &  45 &    &    &    &    &  41 \\\hline
    \end{array}$
   \end{footnotesize}
   \end{center}
\end{ex}

\begin{prop}\label{biembk}
Let $k$ be an odd integer and let
$A$ be a cyclically $k$-diagonal partially filled square
array of size $n\geq k$ such that its nonempty entries are pairwise distinct.
If $\gcd(n,k-1)=1$, then there exist two compatible orderings $\omega_r$ and $\omega_c$ of the  rows and the columns of
$A$.
\end{prop}

\begin{proof}
It is not restrictive to consider $A$ written in the standard form, so that its  nonempty entries are the diagonals
$D_1,\dots, D_{k}$.
Let $\omega_r$ be
the natural ordering of the rows of $A$ from left to right and let $\omega_c$ be
the natural ordering of the columns of $A$ from top to bottom for the first $n-1$ columns, and from bottom to top for
the last column, namely:
$$\begin{array}{rcl}
\omega_r & =&
(a_{1,1},a_{1,2},\ldots,a_{1,k})(a_{2,2},a_{2,3}\ldots,a_{2,k+1})\cdots(a_{n,n},a_{n,1}\ldots,a_{n,k-1}),\\
\omega_c & =& (a_{n-k+2,1},a_{n-k+3,1},\ldots,a_{n,1},a_{1,1})\\
&&            (a_{n-k+3,2},a_{n-k+4,2},\ldots,a_{n,2},a_{1,2},a_{2,2})\cdots \\
&& (a_{n,k-1}, a_{1,k-1},a_{2,k-1},\ldots,a_{k-1,k-1}) (a_{1,k},a_{2,k},\ldots,a_{k,k})\\
&&   (a_{2,k+1},a_{3,k+1},\ldots,a_{k+1,k+1})\cdots(a_{n-k,n-1},a_{n-k+1,n-1},\ldots,a_{n-1,n-1})\\
&&(a_{n,n},a_{n-1,n},\ldots,a_{n-k+1,n}).
  \end{array}$$
Then $\omega_r\circ \omega_c$ moves cyclically from left to right and goes down $n-1$ times and up once. Setting
$t=n-k+1$, we obtain that
$$\omega_r\circ \omega_c
=(D'_2,D'_4,\ldots,D'_{k-1},D'_1,D'_3,\ldots,D'_{k-2},
a_{t,k+t-1},a_{2t,k+2t-1},\ldots,
a_{nt,k+nt-1}),$$
where the indices are considered modulo $n$,  and $D'_s$ is the sequence
$$ a_{n-s+1,n}, a_{n-s+2,1},\ldots,a_{n,s-1}, a_{1,s},a_{2,s+1},\ldots,a_{n-s,n-1}.$$
Note that, for $s\in \{1,\dots,k-1\}$, the elements of $D'_s$ are exactly  the ones of the diagonal $D_s$ and hence are
pairwise distinct.
The last $n$ elements $a_{t,k+t-1},a_{2t,k+2t-1},\ldots,$
$a_{nt,k+nt-1}$ belong to the diagonal $D_k$ and are pairwise distinct
since $\gcd(n,t)=1$.
Therefore $\omega_r\circ \omega_c$ is a cycle of order $nk$.
\end{proof}

\begin{ex}\label{ex3}
Considering the array $B$ of Example \ref{ex2} we have the following ordering for the rows:
   $$\begin{array}{rcl}
   \omega_r & =& (1,2,3,4,5)(6,7,8,9,10)(11,12,13,14,15)(16,17,18,19,20)\\
   &&(21,22,23,24,25)(26,27,28,29,30)
   (31,32,33,34,35)(36,37,38,39,40)\\
   &&(41,42,43,44,45);
   \end{array}$$
and the following ordering for the column:
      $$\begin{array}{rcl}
   \omega_c & =&
(30,34,38,42,1)(35,39,43,2,6)(40,44,3,7,11)(45,4,8,12,16)\\
&&(5,9,13,17,21)(10,14,18,22,26)(15,19,23,27,31)
(20,24,28,32,36)\\
&&(41,37,33,29,25).
      \end{array}$$
Hence,
$$\begin{array}{rcl}\omega_r\circ \omega_c&=&(37,
42,2,7,12,17,22,27,32,29,34,39,44,4,9,14,19,24,41,
1,6,\\
&&11, 16,21,26,31,36,33,38,43,3,8,13,18,23,28,25,5,30,10,35,\\
&&15,40,20, 45).\end{array}$$
\end{ex}

\begin{prop}\label{biemb7}
Let $A$ be a cyclically $7$-diagonal partially filled square
array of odd size $n\geq 7$, such that its nonempty entries are pairwise distinct.
Then there exist two compatible orderings $\omega_r$ and $\omega_c$ of the  rows and the columns of $A$.
\end{prop}

\begin{proof}
It is not restrictive to consider $A$ written in the standard form, so that its  nonempty entries are the diagonals
$D_1,\dots, D_{7}$.
Let $\omega_r$ be
the natural ordering of the rows of $A$ from left to right and let $\omega_c$ be
the natural ordering of the columns of $A$ from top to bottom for the first $n-4$ columns, and from bottom to top for
the last $4$ columns,
that is:
$$\begin{array}{rcl}
\omega_r & =& (a_{1,1},a_{1,2},\ldots,a_{1,7})(a_{2,2},a_{2,3}\ldots,a_{2,8})\cdots(a_{n,n},a_{n,1}\ldots,a_{n,6}),\\
\omega_c & =& (a_{n-5,1},a_{n-4,1},\ldots,a_{n,1},a_{1,1})        
(a_{n-4,2},a_{n-3,2},\ldots,a_{n,2},a_{1,2},a_{2,2})\cdots \\
&& (a_{n,6}, a_{1,6},a_{2,6},\ldots,a_{6,6}) (a_{1,7},a_{2,7},\ldots,a_{7,7})  
(a_{2,8},a_{3,8},\ldots,a_{8,8})\cdots\\
&&(a_{n-10,n-4},a_{n-9,n-4},\ldots,a_{n-4,n-4})
(a_{n-3,n-3},a_{n-4,n-3},\ldots,a_{n-9,n-3})\\
&&(a_{n-2,n-2},a_{n-3,n-2},\ldots,a_{n-8,n-2})
(a_{n-1,n-1},a_{n-2,n-1},\ldots,a_{n-7,n-1})\\
&&(a_{n,n},a_{n-1,n},\ldots,a_{n-6,n}).
  \end{array}$$
Then $\omega_r\circ \omega_c$ moves cyclically from left to right and goes down $n-4$ times and up four times.
It can be showed that $\omega_r\circ \omega_c$ is a cycle of order $7n$. However, since the proof depends on the residue
class of $n$ modulo $6$, we present here only the case $n\equiv 3\pmod 6$, i.e. the case not covered by
Proposition \ref{biembk} (for $n=9$ it suffices an easy  direct check, so we also assume $n>9$).

For $s=1,\ldots,6$, consider the sequences
$$D'_s= \quad a_{n-s+1,n},a_{n-s+2,1},\ldots,a_{n,s-1}, a_{1,s},a_{2,s+1},\ldots, a_{n-3-s,n-4}$$
and
$$
E_s=\left\{\begin{array}{ll}
 a_{n-9+s, n-3+s}, a_{n-9+s-6, n-3+s-6},\ldots, a_{s,6+s},   a_{n-6+s,s} & \textrm{ if } s=1,2,3; \\
 a_{n-15+s, n-9+s}, a_{n-15+s-6, n-9+s-6},\ldots, a_{s,6+s},   a_{n-6+s,s} & \textrm{ if } s=4,5,6. \\
\end{array}\right.$$
Then, it is easy to see that
$$\begin{array}{rcl}
\omega_r\circ \omega_c
&=&(D'_4, a_{n-8,n-3},a_{n-2,n-2},a_{n-3,n-1},D'_5, E_6 ,D'_2,a_{n-6,n-3},a_{n-7,n-2},\\
&& a_{n-1,n-1}, D'_3, a_{n-7,n-3}, E_1,  E_4,  a_{n-3,n-2},a_{n-4,n-1}, D'_6, a_{n-3,n-3},\\
&& a_{n-4,n-2}, a_{n-5,n-1},  E_3, a_{n-4,n-3},a_{n-5,n-2},a_{n-6,n-1}, D'_1,a_{n-5,n-3}, \\
&&a_{n-6,n-2},E_2, E_5 ,a_{n-2,n-1}).
\end{array}$$
Since that the elements of $D'_s$ are those of $D_s\setminus\{a_{n-2-s,n-3},a_{n-1-s,n-2},a_{n-s,n-1}\}$
for all $s=1,\ldots,6$ and the elements of $E_1\cup \ldots \cup E_6$ are those of $D_7$, it follows that
$\omega_r\circ \omega_c$ is a cycle of order $7n$.
\end{proof}

\section{Direct constructions of $\SH^*(n;k)$ for $k\leq 10$}\label{sec4}

In this section, we provide direct constructions of $\SH^*(n;k)$ for $6\leq k\leq 10$ and for any  $n$ satisfying the necessary conditions of Theorem \ref{th:integer}.

Clearly, the main task is to check the simplicity of each row and each column.
A little help is given by noticing that, from Definition \ref{def:H}, the $i$-th partial sum
$s_i$ is different from $s_{i+1}$ and from $s_{i+2}$ both modulo $2nk+1$ and modulo $2nk+2$, where the subscripts are taken modulo $k$.
So, if $k=3,4,5$ then every ordering of any row and column of an $H=\H(n;k)$ is simple both modulo $2nk+1$ and modulo $2nk+2$,
and hence $H$ is an $\SH^*(n;k)$.
We recall that, for these values of $k$, explicit constructions of $\H(n;k)$'s have been described in \cite{ADDY,DW}.
So, we start with the case $k=6$.

We also fix some notation. Given a row or a column $A$ of a partially filled array, we denote by $\nn{A}$ the list of the absolute values of the nonempty entries of $A$ and
by $\S(A)$ the sequence of the partial sums of $A$ with respect to the natural ordering (ignoring the empty cells).
More generally, if $A_1,\ldots,A_r$ are rows (or columns), by $\nn{\cup_{i=1}^r A_i}$ we mean the union $\cup_{i=1}^r \nn{A_i}$.
Furthermore, $\q^t$ means a sequence of $t$ empty cells.

\begin{prop}\label{prop:6}
Let $n\geq 6$ be even. Then, there exists an $\SH^*(n;6)$.
\end{prop}

\begin{proof}
Let $H$ be the $n\times n$ partially filled array whose rows $R_t$ are as follows:
\begin{footnotesize}
\begin{eqnarray*}
R_1 & = & (5, -1, 2, -7, -9, 10, \q^{n-6} ), \\
R_2 & = & (-4, 3, -6, 8, 11, -12, \q^{n-6}),\\
R_{3+2i} &= &(\q^{2+2i}, -13-12i, 17+12i, 14+12i, -19-12i, -21-12i, 22+12i, \q^{n-8-2i}),\\
R_{4+2i} & = &(\q^{2+2i}, 15+12i, -16-12i,-18-12i, 20+12i, 23+12i,-24-12i ,\q^{n-8-2i}),\\
\end{eqnarray*}
 \end{footnotesize}
 \begin{footnotesize}
 \begin{eqnarray*}
R_{n-3} & = &  (-14+6n, 15-6n, \q^{n-6}, 23-6n, -19+6n, -22+6n, 17-6n),\\
R_{n-2} & = & (12-6n, -13+6n, \q^{n-6}, -21+6n, 20-6n, 18-6n, -16+6n),\\
R_{n-1} &= & (5-6n, -10+6n, 3-6n, -2+6n, \q^{n-6}, 11-6n, -7+6n),\\
R_{n} & = & (-4+6n, 6-6n, -1+6n, -6n, \q^{n-6}, -9+6n,8-6n),
\end{eqnarray*}

\end{footnotesize}
\noindent where $i=0,\ldots,\frac{n-8}{2}$. Note that every row contains exactly $6$ filled cells. Also, it is easy
to see that $\nn{R_{2h+1}\cup
R_{2h+2}}=\{1+12h,\ldots, 12+12h\}$ for all $h=0,\ldots,\frac{n-2}{2}$. Hence, $H$ satisfies conditions (a)
and (b) of Definition \ref{def:H}. Now, we list the partial sums for each row. We have
\begin{footnotesize}
\begin{eqnarray*}
\S(R_1) & = &(5,4,6,-1,-10,0), \\
\S(R_2) & = & (-4,-1,-7,1,12,0),\\
\S(R_{3+2i}) & =& ( -13-12i,  4,  18+12i,  -1, -22-12i, 0 ),\\
\S(R_{4+2i}) & =& (15+ 12i, -1,  -19 -12i, 1, 24+12i, 0 ),\\
\S(R_{n-3}) & = & (-14+6n, 1, 24-6n, 5, -17+6n, 0),\\
\S(R_{n-2}) & = & (12-6n, -1, -22+6n, -2, 16-6n,0 ), \\
\S(R_{n-1}) &= & (5-6n, -5, -2-6n, -4, 7-6n, 0),\\
\S(R_{n}) & = & (-4+6n, 2, 1+6n, 1, -8+6n, 0).
\end{eqnarray*}

\end{footnotesize}
\noindent Clearly, each row sums to $0$.
In view of previous considerations, in order to prove that each row is simple modulo $\nu\in \{12n+1, 12n+2\}$ it suffices to prove that
$s_{i}\not \equiv s_{i+3} \pmod \nu$ for $i=1,2,3$. From the definition of $H$ we obtain the following  expression of the columns:
\begin{footnotesize}
\begin{eqnarray*}
C_1 & = & (5, -4, \q^{n-6}, -14+6n, 12-6n, 5-6n, -4+6n)^T,\\
C_2 & = & (-1, 3, \q^{n-6}, 15-6n, -13+6n,-10+6n,6-6n  )^T, \\
C_3 & = & (2, -6, -13, 15, \q^{n-6}, 3-6n, -1+6n )^T,\\
C_4 & = & (-7,8,17,-16,\q^{n-6}, -2+6n, -6n)^T,\\
C_{5+2j} & = & (\q^{2j}, -9-12j, 11+12j, 14+12j, -18-12j, -25-12j, 27+12j, \q^{n-6-2j})^T,\\
C_{6+2j} & = & (\q^{2j}, 10+12j, -12-12j, -19-12j, 20+12j, 29+12j,-28-12j , \q^{n-6-2j})^T,
\end{eqnarray*}

\end{footnotesize}
\noindent where $j=0,\ldots,\frac{n-6}{2}$. We observe that each column contains exactly $6$ filled cells, then
condition (c) of Definition \ref{def:H} is satisfied.
The lists of the partial sums of the columns are
\begin{footnotesize}
\begin{eqnarray*}
\S(C_1) & = & (5, 1, -13+6n, -1, 4-6n, 0),\\
\S(C_2) & = & (-1, 2, 17-6n, 4, -6+6n, 0), \\
\S(C_3) & = & (2, -4, -17, -2, 1-6n, 0),\\
\S(C_4) & = & (-7,1, 18, 2, 6n, 0),\\
\S(C_{5+2j}) & = & (-9-12j, 2, 16+12j, -2, -27-12j,0),\\
\S(C_{6+2j}) & = & (10+12j, -2, -21-12j, -1, 28+12j, 0).
\end{eqnarray*}

\end{footnotesize}
\noindent Since every column sums to $0$, also condition (d)  of Definition \ref{def:H}  holds and so $H$ is an
$\H(n;6)$. Finally, for each column
one can check that
$s_{i}\not \equiv s_{i+3} \pmod \nu$, for $i=1,2,3$, where $\nu \in \{12n+1,12n+2\}$.
So we conclude that $H$ is a globally simple $\SH(n;6)$ that also satisfies condition $(\ast)$, namely $H$ is an $\SH^*(n;6)$.
\end{proof}

\begin{ex}
Following the construction illustrated in the proof of Proposition \ref{prop:6} for $n=8$, we obtain

\begin{footnotesize}
$$\SH^*(8;6)=\begin{array}{|r|r|r|r|r|r|r|r|}\hline
5 & -1 & 2 & -7 & -9 & 10 &  &  \\\hline
-4 & 3 & -6 & 8 & 11 & -12 &  &  \\\hline
 &  & -13 & 17 & 14 & -19 & -21 & 22 \\\hline
 &  & 15 & -16 & -18 & 20 & 23 & -24 \\\hline
34 & -33 &  &  & -25 & 29 & 26 & -31 \\\hline
-36 & 35 &  &  & 27 & -28 & -30 & 32 \\\hline
-43 & 38 & -45 & 46 &  &  & -37 & 41 \\\hline
44 & -42 & 47 & -48 &  &  & 39 & -40 \\\hline
\end{array}$$
\end{footnotesize}
\end{ex}

\begin{prop}\label{SH7_even}
Let $n\equiv 0 \pmod 4$ and $n\geq 8$. Then, there exists an $\SH^*(n;7)$.
\end{prop}

\begin{proof}
An $\SH^*(8;7)$ and an $\SH^*(12;7)$ are given in \cite{web}. So,  assume $n=4a\geq 16$ and let $H$ be the $n\times n$ array whose rows $R_t$ are the following:
 \begin{footnotesize}
\begin{eqnarray*}
R_1 & = & ( -5a,\q^{2a-8}, 3-24a,\q, -3+16a,\q^3, -1-2a, \q^{2a-6}, -4+24a, \q, 4-16a,\\
&&\q, 1+7a,\q ),\\
R_2 & = & (\q, -2-6a, \q^{2a-8}, 5-28a, \q, -5+20a, \q^3, 2, \q^{2a-6}, -6+28a,\q, 6-20a, \\
&& \q, 6a ),\\
R_3 & = & ( 2+7a, \q, 1-5a,\q^{2a-8}, 7-24a, \q, -7+16a,\q^3, -3-2a,\q^{2a-6}, -8+24a,\\
&&\q, 8-16a,\q ),\\
R_4 & = & (\q, -1+6a, \q ,-3-6a,\q^{2a-8}, 9-28a, \q, -9+20a,\q^3, 4,\q^{2a-6}, -10+ 28a,\\
&&\q, 10-20a),\\
 R_5 & = & (  12-16a, \q, 3+7a, \q, 2-5a,\q^{2a-8}, 11-24a, \q, -11+16a,\q^3, -5-2a,\\
&&\q^{2a-6}, -12+24a,\q ),\\
 R_6 & = & (\q, 14-20a, \q, -2+6a, \q, -4-6a,\q^{2a-8}, 13-28a,\q, -13+ 20a,\q^3, 6,\\
&&\q^{2a-6}, -14+28a ),\\
R_{7+2i} &= & (\q^{2i}, -16+24a-4i, \q,  16-16a+4i,\q,  4+7a+i, \q, 3-5a+i,\q^{2a-8}, \\
&&  15-24a+4i, \q, -15+16a-4i, \q^3,  -7-2a-2i , \q^{2a-6-2i}),\\
R_{8+2i} & =& (\q^{1+2i}, -18+28a-4i,\q, 18-20a+4i,\q,  -3+6a-i, \q, -5-6a-i,\q^{2a-8}, \\
&& 17-28a+4i, \q, -17+20a-4i,\q^3, 8+2i ,\q^{2a-7-2i}),\\
R_{2a-1} & = &( \q^{2a-8},28a, \q, -20a, \q, 8a,\q,  -12a,\q^{2a-8}, -1-20a, \q, 1+12a, \q^3, 4a, \q^2),\\
R_{2a} & =& (\q^{2a-7},  -2+24a, \q, 2-16a, \q, 1+5a,\q, -9a,\q^{2a-8}, 1-24a, \q, -1+16a, \\
&& \q^3, -1+4a,\q),\\
R_{2a+1} & = & (\q^{2a-6}, -4+28a, \q, 4-20a,\q, 10a, \q, -1 -10a, \q^{2a-8},  3 -28a, \q, -3+ 20a,  \\
&&\q^3, 1),\\
R_{2a+2} & = & ( -2-2a,\q^{2a-6}, -6+24a, \q, 6-16a,\q, 1+ 11a, \q,  1-9a,\q^{2a-8}, 5-24a, \\
&&\q, -5+16a, \q^3),\\
R_{2a+3} & = & (\q, 3,\q^{2a-6}, -8+28a, \q, 8-20a, \q, -1+10a,\q,  -2-10a,\q^{2a-8},  7-28a, \\
&&\q, -7+20a,\q^2),\\
R_{2a+4} & = & (\q^2, -4-2a,\q^{2a-6}, -10+24a, \q, 10-16a,\q,  2+11a, \q,  2-9a,\q^{2a-8},\\
&& 9-24a, \q,  -9+16a,\q ),\\
R_{2a+5} & =& (\q^3, 5,\q^{2a-6}, -12+28a,\q, 12-20a,\q, -2+10a,\q, -3-10a,\q^{2a-8},\\
&& 11-28a,\q, -11+20a ),\\
R_{2a+6} & = & ( -13+16a,\q^3, -6 -2a,\q^{2a-6},-14+ 24a,\q, 14-16a,\q, 3+11a,\q, 3-9a,\\
&&\q^{2a-8}, 13-24a,\q ),\\
R_{2a+7} & =& (\q, -15+20a, \q^3, 7,\q^{2a-6}, -16+28a,\q, 16-20a,\q, -3+10a,\q, -4-10a,\\
&&\q^{2a-8}, 15-28a ),\\
\end{eqnarray*}
 \end{footnotesize}
 \begin{footnotesize}
 \begin{eqnarray*}
R_{2a+8+2i} & = &( \q^{2i}, 17-24a+4i, \q,  -17+16a-4i,\q^3, -8-2a-2i ,\q^{2a-6},\\
&& -18+24a-4i, \q,   18-16a+4i, \q,  4+11a+i,  \q,  4-9a+i,\q^{2a-8-2i}),\\
R_{2a+9+2i} & = & (\q^{1+2i}, 19-28a+4i, \q, -19+20a-4i,\q^3, 9+2i,\q^{2a-6},    -20+28a-4i, \\
   &&\q, 20-20a+4i, \q, -4+10a-i,\q,  -5-10a-i ,\q^{2a-9-2i}),\\
R_{4a} & = & (\q^{2a-8}, 1-28a, \q, -1+20a, \q^3,2a,\q^{2a-6}, -2+28a, \q,  2-20a,\\
&&\q,  1+4a,\q, -1-6a ),
\end{eqnarray*}

\end{footnotesize}
\noindent where $i=0,\ldots,a-5$. Note that each row contains exactly $7$ elements.
It is not hard to see that
\begin{footnotesize}
\begin{gather*}
\begin{array}{rcl}
\nn{\cup_{i=0}^{a-2} R_{2+2i}\cup \cup_{j=0}^{a} R_{2a-1+2j}\cup R_{4a}} & =&
\{1,\ldots,2a\}\cup \{4a,4a+1\}\cup
\{5a+2,\ldots, 7a\}\cup \\
&&\{8a\}\cup \{9a+1,\ldots,11a \}\cup \{12a,12a+1\}\cup\\
&&\{16a,\ldots,20a+1\}\cup \{24a,\ldots,28a\},
\end{array}\\
\begin{array}{rcl}
\nn{\cup_{i=0}^{a-2} R_{1+2i}\cup \cup_{j=0}^{a-1} R_{2a+2j}} &
=& \{2a+1,\ldots,4a-1\}\cup\{4a+2,\ldots,5a+1\}\cup \{7a+1,\ldots,\\
&& 8a-1\}\cup \{8a+1,\ldots,9a\}\cup \{11a+1,\ldots,12a-1 \}\cup\\
&& \{12a+2,\ldots,16a-1\}\cup \{20a+2,\ldots,24a-1\}.
\end{array}
\end{gather*}

\end{footnotesize}
\noindent Hence, $H$ satisfies conditions (a) and (b) of Definition \ref{def:H}.
Now, we list the partial sums for each row. We have

\begin{footnotesize}
\begin{eqnarray*}
\S(R_{1}) & =  & ( -5a, 3-29a, -13a, -1-15a, -5+9a, -1-7a, 0 ), \\
\S(R_{2}) & =  & ( -2-6a, 3-34a, -2-14a, -14a, -6+14a, -6a, 0), \\
\S(R_{3}) & =  & ( 2+7a, 3+2a, 10-22a, 3-6a, -8a, -8+16a, 0 ),\\
 \S(R_{4}) & =  & ( -1+6a, -4, 5-28a, -4-8a, -8a, -10+20a, 0 ), \\
 \S(R_{5}) & =  & ( 12-16a, 15-9a, 17-14a, 28-38a, 17-22a, 12-24a, 0 ), \\
\S(R_{6}) & =  & ( 14 -20a, 12-14a, 8-20a, 21-48a, 8-28a, 14-28a, 0 )\\
\S(R_{7+2i}) & = & ( -16+24a-4i, 8a, 4+15a+i, 7+10a+2i, 22-14a+6i, 7+2a+2i, 0),\\
\S(R_{8+2i}) & = & ( -18+28a-4i, 8a, -3+14a-i, -8+8a-2i, 9-20a+2i, -8-2i, 0 ),\\
\S(R_{2a-1}) & = & ( 28a, 8a, 16a, 4a, -1-16a, -4a, 0  ),\\
\S(R_{2a}) & = & ( -2+24a, 8a, 1+13a, 1+4a, 2-20a, 1-4a, 0 ),\\
\S(R_{2a+1}) & = & ( -4+28a, 8a, 18a, -1+8a, 2-20a, -1, 0  ),\\
\S(R_{2a+2}) & = & ( -2-2a, -8+22a, -2+6a, -1+17a, 8a, 5-16a, 0 ),\\
\S(R_{2a+3}) & = & ( 3, -5+28a, 3+8a, 2+18a, 8a, 7-20a, 0),\\
\S(R_{2a+4}) & = & ( -4-2a, -14+22a, -4+6a, -2+17a, 8a, 9-16a, 0 ),\\
\S(R_{2a+5}) & = & ( 5, -7+28a, 5+8a, 3+18a, 8a, 11-20a, 0 ),\\
\S(R_{2a+6}) & = & ( -13+16a, -19+14a, -33+38a, -19+22a, -16+33a, -13+24a, 0 ),\\
\S(R_{2a+7}) & = & ( -15+20a, -8+20a, -24+48a, -8+28a, -11+38a, -15+28a, 0 ),\\
\S(R_{2a+8+2i}) & = & (17-24a+4i, -8a, -8-10a-2i, -26+14a-6i, -8-2a-2i,\\
&&-4+9a-i, 0),\\
\S(R_{2a+9+2i}) & = & ( 19-28a+4i, -8a, 9-8a+2i, -11+20a-2i, 9+2i, 5+10a+i, 0 ),\\
\S(R_{4a}) & = & ( 1-28a, -8a, -6a, -2+22a, 2a, 1+6a, 0).\\
\end{eqnarray*}

\end{footnotesize}
\noindent Clearly, each row sums to $0$.
By a direct check (keeping in mind previous considerations on partial sums) one can see that the elements of each $\S(R_t)$ are pairwise distinct modulo $14n+1$ and modulo $14n+2$ for any
$n\equiv 0 \pmod 4$.
From the definition of $H$ we obtain the following expression of the columns:

\begin{footnotesize}
\begin{eqnarray*}
C_1 &=& ( -5a,\q, 2+7a,\q, 12-16a,\q, -16+24a,\q^{2a-6}, -2-2a,\q^3, -13+16a, \q, \\
&&17-24a, \q^{2a-8} )^T,\\
C_{2+2i} &= & (\q^{1+2i},  -2-6a-i,\q, -1+6a-i,\q, 14-20a+4i,\q, -18+28a-4i,\q^{2a-6},\\
&&3+2i, \q^3,  -15+20a-4i,\q, 19-28a+4i,\q^{2a-9-2i} )^T,\\
C_{3+2h} & =& (\q^{2+2h},  1-5a+h, \q,3+7a+h,\q,16-16a+4h,\q, -20+24a-4h, \q^{2a-6},\\
&&-4-2a-2h, \q^3,  -17+16a-4h,\q, 21-24a+4h,\q^{2a-10-2h})^T,\\
C_{2a-7} & =& (\q^{2a-8}, -4-4a, \q, -2+8a, \q, -4-12a, \q, 28a,\q^{2a-6}, 6-4a,\q^3, 3+12a,\\
&& \q,  1-28a)^T,\\
C_{2a-6} & = &(3-24a,\q^{2a-8},2-7a,\q,3+5a,\q,-2-16a,\q,-2+24a, \q^{2a-6}, -5+2a,\\
&&\q^3, 1+16a,\q )^T,\\
C_{2a-5} & =& (\q,5-28a, \q^{2a-8},-3-4a,\q,-1+8a,\q,-20a, \q,-4+28a, \q^{2a-6}, 4-4a ,\\
&&\q^3, -1+20a)^T,\\
C_{2a-4} & = & (-3+16a,\q,7-24a, \q^{2a-8},1-7a,\q,2+5a,\q,2-16a,\q, -6+24a,\q^{2a-6},\\
&&-3+2a,\q^3)^T,\\
C_{2a-3} & =& (\q,-5+20a,\q,9-28a,\q^{2a-8},-2-4a,\q,8a,\q,4-20a, \q,-8+28a, \q^{2a-6},\\
&& 2-4a,\q^2)^T,\\
C_{2a-2} & =& (\q^2,-7+16a,\q,11-24a,\q^{2a-8},-7a,\q ,1+5a,\q ,6-16a, \q,-10+24a,\\
&& \q^{2a-6}, -1+2a,\q)^T,\\
C_{2a-1} & =& (\q^3,  -9+20a,\q,13-28a, \q^{2a-8},-12a,\q,10a,\q,8-20a, \q,-12+28a,\\
&&\q^{2a-6}, 2a)^T,\\
C_{2a+2j} & = &  (\q^{2j},-1-2a-2j,\q^3, -11+16a-4j,\q,15-24a+4j,\q^{2a-8},-9a+j,\q,\\
&&1+11a+j,\q,10-16a+4j,\q,-14+24a-4j, \q^{2a-6-2j} )^T,\\
C_{2a+1+2j} & =& (\q^{1+2j},2+2j,\q^3,-13+20a-4j,\q,17-28a+4j,\q^{2a-8},-1-10a-j,\\
&&\q,-1+10a-j,\q, 12-20a+4j,\q,-16+28a-4j,\q^{2a-7-2j})^T,\\
C_{4a-6} & =& (\q^{2a-6},5-4a, \q^3, 1+12a, \q, 3-28a, \q^{2a-8},-3-8a, \q, -2+12a, \q, \\
&&-2-12a, \q,
-2+28a )^T,\\
C_{4a-5} & =& (-4+24a,\q^{2a-6},-4+2a,\q^3,-1+16a,\q,5-24a,\q^{2a-8}, 2-11a,\q, \\
&&2+9a,\q,-16a,\q)^T,\\
C_{4a-4} & =& (\q,-6+28a,\q^{2a-6},3-4a,\q^3, -3+20a,\q,7-28a,\q^{2a-8},-2-8a, \q, \\
&&-1+12a,\q,2-20a)^T,\\
C_{4a-3} & = & ( 4-16a,\q,-8+24a,\q^{2a-6},-2+2a,\q^3,-5+16a,\q,9-24a,\q^{2a-8},1-11a,\\
&&\q,1+9a,\q)^T,\\
C_{4a-2} & =& ( \q,6-20a,\q,-10+28a,\q^{2a-6},4a,\q^3,-7+20a,\q,11-28a, \q^{2a-8}, -1-8a,\\
&&\q,1+4a)^T,\\
C_{4a-1} & =& ( 1+7a,\q,8-16a,\q,-12+24a,\q^{2a-6},-1+4a,\q^3,-9+16a,\q, 13-24a,\\
&&\q^{2a-8},-11a,\q)^T,\\
C_{4a} & =& (\q,6a,\q,10-20a,\q,-14+28a, \q^{2a-6},1,\q^3,-11+20a,\q ,15-28a,\\
&&\q^{2a-8},-1-6a)^T,
 \end{eqnarray*}

\end{footnotesize}
\noindent for $i=0,\ldots,a-5$, $h=0,\ldots,a-6$ and $j=0,\ldots,a-4$. Note that each column contains exactly
$7$ elements, hence $H$ satisfies also condition (c).
One can check that the partial sums for the columns are the following:
\begin{footnotesize}
\begin{eqnarray*}
\S(C_1)  & = & (          -5a,   2+2a,  14-14a,  -2+  10a,   -4+  8a,  -17+ 24a,     0 ),\\
\end{eqnarray*}
 \end{footnotesize}
 \begin{footnotesize}
 \begin{eqnarray*}
\S(C_{2+2i})  & = & (  -2-6a-i,   -3 -2i,  11-20a+ 2i,  -7+8a  -2i,   -4+  8a, \\
&&-19+28a -4i,     0 ),\\
\S(C_{3+2h}) & = & (  1-5a+h, 4+2a+2h, 20 -14a+6h,  10a+2h,   -4+  8a,\\
&&-21+ 24a-4h,     0 ),\\
\S(C_{2a-7})  & = & (  -4  -4a,   -6+  4a,  -10 -8a,  -10+ 20a,  -4+  16a,   -1+ 28a,     0 ),\\
\S(C_{2a-6})  & = & (   3-24a,  5 -31a,  8 -26a,  6 -42a,  4 -18a,  -1 -16a,     0 ),\\
\S(C_{2a-5}) & = & (  5 -28a,  2 -32a,  1 -24a,  1 -44a,  -3 -16a,  1 -20a,     0 ),\\
\S(C_{2a-4}) & = & (    -3+16a,  4  -8a,  5 -15a,  7 -10a,  9 -26a,  3  -2a,     0 ),\\
\S(C_{2a-3})  & = & (   -5+ 20a,   4 -8a,  2 -12a,   2 -4a,  6 -24a,  -2+   4a,     0 ),\\
\S(C_{2a-2})  & = & (  -7+  16a,   4 -8a,  4 -15a,  5 -10a, 11 -26a,   1 -2a,     0 ),\\
\S(C_{2a-1})  & = & (   -9+ 20a,   4 -8a,  4 -20a,  4 -10a, 12 -30a,  -2a,     0 ),\\
\S(C_{2a+2j})  & = & (  -1  -2a-2j,   -12+14a-6j,   3-10a-2j,  3-19a-j,   4 -8a, \\
&&14 -24a+4j,     0 ),\\
 \S(C_{2a+1+2j}) & = & (    2+ 2j, -11+  20a-2j,   6 -8a+2j, 5-18a+j,   4 -8a,  16-28a+4j,     0 ),\\
\S(C_{4a-6})  & = & (   5 -4a,   6+  8a,  9 -20a,  6 -28a,  4 -16a,  2 -28a,     0 ),\\
 \S(C_{4a-5}) & = & (   -4+ 24a,   -8+ 26a,  -9+  42a,  -4+  18a,   -2+  7a,  16a,     0 ),\\
\S(C_{4a-4}) & = & (  -6+  28a,   -3+ 24a,  -6+  44a,  1+  16a,   -1+  8a,  -2 + 20a,     0 ),\\
 \S(C_{4a-3}) & = & (   4-16a,   -4+  8a,  -6+  10a,  -11+ 26a,   -2+  2a,   -1 -9a,     0 ),\\
\S(C_{4a-2})  & = & (   6-20a,   -4+  8a,  -4+  12a, -11+  32a,   4a,   -1 -4a,     0 ),\\
\S(C_{4a-1}) & = & (   1+  7a,  9  -9a,  -3+  15a,   -4+ 19a,  -13+ 35a,  11a,     0 ),\\
 \S(C_{4a}) & = & (   6a, 10 -14a,  -4+  14a,  -3+  14a,  -14+ 34a,   1+  6a,     0).
\end{eqnarray*}

\end{footnotesize}
\noindent Note that each column sums to $0$  and so condition (d) is satisfied, hence $H$ is an
$\H(n;7)$.
By a direct check one can verify that the elements of each $\S(C_t)$ are pairwise distinct modulo $14n+1$ and modulo $14n+2$ for any $n\equiv 0 \pmod 4$.
Thus, for these values of $n$, $H$ is an $\SH^*(n;7)$.
\end{proof}

\begin{prop}\label{SH7_odd}
Let $n\equiv 1 \pmod 4$ and $n\geq 9$. Then, there exists an $\SH^*(n;7)$.
\end{prop}

\begin{proof}
Assume $n=4a+1\geq 9$ and let $H$ be the $n\times n$ array whose rows $R_t$ are the following:
  \begin{footnotesize}
\begin{eqnarray*}
R_1 & = &( -1-4a, -5-16a, -2-7a, 5+12a,\q^{4a-6}, 6+28a, 3+11a, -6-24a ),\\
R_2 & =& ( -5-18a, -4a, -7-22a, -3-8a, 6+14a,\q^{4a-6}, 6+26a, 3+12a),\\
R_3 & =& ( 3+9a, -6-20a, -1+2a, -7-20a, -2-11a, 6+12a,\q^{4a-6},7+28a ),\\
R_{4+4i} & =& (\q^{4i}, 7+26a+2i, 2+8a-i, -6-18a-2i, 1-4a+2i, -8-22a-2i, \\
&&-3-4a-i, 7+14a+2i ,\q^{4a-6-4i}),\\
R_{5+4i} & =& (\q^{1+4i}, 7+24a+2i, 3+5a+i, -6-16a-2i, -2+2a-2i, -8-20a-2i, \\
&&-1-7a+i, 7+12a+2i ,\q^{4a-7-4i}),\\
R_{6+4i} & = & (\q^{2+4i}, 8+26a+2i, 2+12a-i, -7-18a-2i, 2-4a+2i, -9-22a-2i, \\
&&-4-8a-i, 8+14a+2i ,\q^{4a-8-4i}),\\
R_{7+4j} & = &(\q^{3+4j}, 8+24a+2j, 4+9a+j, -7-16a-2j, -3+2a-2j, -9-20a-2j,\\
&&-1-11a+j, 8+12a+2j ,\q^{4a-9-4j}),\\
R_{4a-1} &= & ( 4+14a, \q^{4a-6},4+26a, 2+10a, -3-18a, 1, -5-22a, -3-10a ),\\
R_{4a} & = & ( -2-5a, 4+12a,\q^{4a-6}, 5+28a, 3+7a, -4-20a, -1-2a, -5-20a ),\\
R_{4a+1}& =& ( -6-22a, -2-4a, 5+14a,\q^{4a-6}, 5+26a, 2+6a, -4-18a, -2a),
\end{eqnarray*}

\end{footnotesize}
\noindent where $i=0,\ldots,a-2$ and $j=0,\ldots,a-3$. Note that each row contains exactly $7$ elements.
One can see that

\begin{footnotesize}
\begin{gather*}
\begin{array}{rcl}
\cup_{i=-1}^ 1 \nn{ R_{2+i}\cup R_{4a+i} } & =& \{1,2a-1,2a,2a+1,4a,4a+1,4a+2, 5a+2,6a+2, 7a+2,\\
&& 7a+3,8a+3, 9a+3, 10a+2,10a+3, 11a+2,11a+3\}  \cup\\
&& \{12a+3, \dots,12a+6\}\cup \{14a+4,14a+5,  14a+6,16a+5, \\
&&18a+3,18a+4,18a+5\}\cup  \{20a+4,\dots, 20a+7\}\cup \{22a+5, \\
&& 22a+6,22a+7,24a+6,26a+4,26a+5,  26a+6,28a+5,\\
&& 28a+6,28a+7 \},
\end{array}\\
\begin{array}{rcl}
\nn{\cup_{i=0}^{4a-6} R_{4+i} } & =& \{2,\ldots,2a-2\}\cup \{2a+2,\ldots,4a-1\}\cup\{4a+3,\ldots,5a+1\}\cup \\
&& \{5a+3,\ldots,6a+1\}\cup \{6a+3,\ldots,7a+1\}\cup\{7a+4,\ldots,8a+2\}\cup \\
&&\{8a+4,\ldots,9a+2\}\cup \{9a+4,\ldots,10a+1\}\cup\{10a+4,\ldots,11a+1\}\cup\\
&&\{11a+4,\ldots,12a+2\}\cup \{12a+7,\ldots,14a+3\}\cup \{14a+7,\ldots,16a+4\}\cup \\
&& \{16a+6,\ldots,18a+2\}\cup \{18a+6,\ldots,20a+3\}\cup \{20a+8,\ldots,22a+4\}\cup\\
&& \{22a+8,\ldots,24a+5\}\cup\{24a+7,\ldots,26a+3\}\cup \{26a+7,\ldots,28a+4\}. \\
\end{array}
\end{gather*}

\end{footnotesize}
\noindent Hence, $H$ satisfies conditions (a) and (b) of Definition \ref{def:H}.
Now, we list the partial sums for each row. We have
\begin{footnotesize}
\begin{eqnarray*}
\S(R_{1}) & =  & (  -1-4a, -6-20a, -8-27a, -3-15a, 3+13a, 6+24a, 0  ), \\
\S(R_{2}) & =  & ( -5-18a, -5-22a, -12-44a, -15-52a, -9-38a, -3-12a, 0  ), \\
\S(R_{3}) & =  & (  3+9a, -3-11a, -4-9a, -11-29a, -13-40a, -7-28a, 0  ), \\
\S(R_{4+4i}) & =  & (  7+26a+2i, 9+34a+i, 3+16a-i, 4+12a+i, -4-10a-i, \\
&&-7-14a-2i, 0  ), \\
\S(R_{5+4i}) & =  & (  7+24a+2i, 10+29a+3i, 4+13a+i, 2+15a-i, -6-5a-3i, \\
&&-7-12a-2i, 0  ), \\
\S(R_{6+4i}) & =  & ( 8+26a+2i, 10+38a+i, 3+20a-i, 5+16a+i, -4-6a-i,\\
&&-8-14a-2i, 0 ),\\
\S(R_{7+4j}) & =  & (  8+24a+2j, 12+33a+3j, 5+17a+j, 2+19a-j, -7-a-3j,\\
&&-8-12a-2j, 0  ), \\
\S(R_{4a-1}) & =  & (  4+14a, 8+40a, 10+50a, 7+32a, 8+32a, 3+10a, 0  ), \\
\S(R_{4a}) & =  & (  -2-5a, 2+7a, 7+35a, 10+42a, 6+22a, 5+20a, 0  ), \\
\S(R_{4a+1}) & =  & (  -6-22a, -8-26a, -3-12a, 2+14a, 4+20a, 2a, 0  ).
\end{eqnarray*}

\end{footnotesize}
\noindent Note that each row sums to $0$.
By a direct check one can verify that the elements of each $\S(R_t)$ are pairwise distinct modulo $14n+1$ and modulo $14n+2$ for any $n\equiv 1 \pmod 4$.
From the definition of $H$ we obtain the following  expression of the columns:
\begin{footnotesize}
\begin{eqnarray*}
C_1 & = & (-1-4a, -5-18a, 3+9a, 7+26a,\q^{4a-6},  4+14a, -2-5a,-6-22a)^T,\\
C_2 & = & ( -5-16a,-4a,-6-20a,2+8a,7+24a, \q^{4a-6},  4+12a,-2-4a )^T,\\
C_3 & = & ( -2-7a , -7-22a,-1+2a , -6-18a,3+5a , 8+26a, \q^{4a-6}, 5+14a )^T,\\
C_{4+4i} & = & (\q^{4i},  5+12a +2i,-3-8a-i,-7-20a-2i, 1-4a+2i,-6-16a-2i,   \\
&& 2+12a-i,8+24a+2i,\q^{4a-6-4i})^T,\\
C_{5+4i} & = & (\q^{1+4i},  6+14a+2i, -2-11a+i, -8-22a-2i, -2+2a-2i, -7-18a-2i, \\
&&4+9a+i, 9+26a+2i,
\q^{4a-7-4i})^T,\\
C_{6+4i} & = &( \q^{2+4i},  6+12a+2i, -3-4a-i,-8-20a-2i, 2-4a+2i,-7-16a-2i,\\
&&1+8a-i, 9+24a+2i,\q^{4a-8-4i} )^T,\\
C_{7+4j} & = & (\q^{3+4j},  7+14a+2j, -1-7a+j, -9-22a-2j, -3+2a-2j, -8-18a-2j, \\
&&4+5a+j,
10+26a+2j, \q^{4a-9-4j} )^T,\\
 C_{4a-1} & = & ( 6+28a, \q^{4a-6}, 3+16a, -3-6a, -5-24a, 1, -4-20a, 2+6a)^T,\\
 C_{4a} & = &  ( 3+11a,6+26a ,\q^{4a-6},  3+14a,-2-9a,-5-22a,-1-2a,-4-18a)^T,\\
C_{4a+1} &= &  ( -6-24a, 3+12a, 7+28a ,\q^{4a-6},  4+16a, -3-10a, -5-20a,-2a)^T,
\end{eqnarray*}

\end{footnotesize}
\noindent where $i=0,\ldots,a-2$ and $j=0,\ldots,a-3$. Note that each column contains exactly $7$ elements, hence $H$
satisfies also condition (c).
One can check that the partial sums for the columns are the following:
\begin{footnotesize}
\begin{eqnarray*}
\S(C_{1}) & =  & ( -1-4a, -6-22a, -3-13a, 4+13a, 8+27a, 6+22a, 0  ), \\
\S(C_{2}) & =  & ( -5-16a, -5-20a, -11-40a, -9-32a, -2-8a, 2+4a, 0  ), \\
\S(C_{3}) & =  & ( -2-7a, -9-29a, -10-27a, -16-45a, -13-40a, -5-14a, 0  ), \\
\S(C_{4+4i}) & =  & (   5+12a+2i,   2+4a +    i,   -5-16a  -i,  -4-20a+    i,  -10-36a  -i,\\
&&-8-24a -2i,          0), \\
\S(C_{5+4i}) & =  & ( 6+14a+  2i,  4+3a+   3i,  -4-19a+    i,   -6-17a  -i, -13-35a -3i,\\
&&-9-26a -2i,            0   ),\\
\S(C_{6+4i}) & =  & ( 6+12a+ 2i,   3+8a+    i,   -5-12a  -i,  -3-16a +   i,  -10-32a  -i, \\
&&-9-24a-2i,
  0  ), \\
\S(C_{7+4j}) & =  & ( 7+14a+2j,  6+  7a+3j,  -3   -15a+j,   -6  -13a-j,  -14-31a-3j,  \\
&&-10-26a-2j,             0 ), \\
\S(C_{4a-1}) & =  & ( 6+28a, 9+44a, 6+38a, 1+14a, 2+14a, -2-6a, 0), \\
\S(C_{4a}) & =  & (  3+11a, 9+37a, 12+51a, 10+42a, 5+20a, 4+18a, 0  ), \\
\S(C_{4a+1}) & =  & ( -6-24a, -3-12a, 4+16a, 8+32a, 5+22a, 2a, 0  ).
\end{eqnarray*}

\end{footnotesize}
\noindent Since each column sums to $0$, condition (d) holds and hence $H$ is an $\H(n;7)$.
By a direct check one can verify that the elements of each $\S(C_t)$ are pairwise distinct modulo $14n+1$ and modulo $14n+2$ for any $n\equiv 1 \pmod 4$.
Thus, for these values of $n$, $H$ is an $\SH^*(n;7)$.
\end{proof}

\begin{ex}
Let $n=9$, by the construction given in the proof of Proposition \ref{SH7_odd}, we obtain the following
$\SH^*(9;7)$:

\begin{footnotesize}
$$
\begin{array}{|c|c|c|c|c|c|c|c|c|}\hline
-9 & -37 & -16 & 29 &  &  & 62 & 25 & -54 \\\hline
-41 & -8 & -51 & -19 & 34 &  &  & 58 & 27 \\\hline
21 & -46 & 3 & -47 & -24 & 30 &  &  & 63 \\\hline
59 & 18 & -42 & -7 & -52 & -11 & 35 &  &  \\\hline
 & 55 & 13 & -38 & 2 & -48 & -15 & 31 &  \\\hline
 &  & 60 & 26 & -43 & -6 & -53 & -20 & 36 \\\hline
32 &  &  & 56 & 22 & -39 & 1 & -49 & -23 \\\hline
-12 & 28 &  &  & 61 & 17 & -44 & -5 & -45 \\\hline
-50 & -10 & 33 &  &  & 57 & 14 & -40 & -4 \\\hline
\end{array}$$
\end{footnotesize}
\end{ex}

\begin{prop}\label{0,2}
Let $n\geq 8$ be even. Then, there exists an $\SH^*(n;8)$.
\end{prop}

\begin{proof}
An $\SH^*(8;8)$ can be found in \cite{web}. So, assume $n \geq 10$.\\
\underline{Case 1.} $n\equiv 0,2 \pmod 6$. Let $H$ be the $n\times n$ array whose rows $R_t$ are the following:
\begin{footnotesize}
\begin{eqnarray*}
R_{1}& = &( -3+8n, 8n, 1-8n, 10-8n, 9-8n, 2-8n,\q^{n-8},-11+8n, -8+8n),\\
R_{2+2i} & =  & (\q^{2+2i}, -17-16i, -20-16i, -25-16i, -28-16i, 27+16i, 18+16i, 19+16i,\\
&&  26+16i, \q^{n-10-2i}), \\
R_{3+2j} & =  & (\q^{2j}, 5+16j, 8+16j, 13+16j, 16+16j, -15-16j,-6-16j,-7-16j,\\
&& -14-16j,\q^{n-8-2j}),\\
R_{n-6}& =& ( -45+8n, -38+8n,\q^{n-8},47-8n, 44-8n, 39-8n, 36-8n, -37+8n,\\
&& -46+8n),\\
R_{n-4}& = &( -21+8n,-30+8n, -29+8n, -22+8n,\q^{n-8},31-8n, 28-8n, 23-8n, \\
&& 20-8n ),\\
\end{eqnarray*}
 \end{footnotesize}
 \begin{footnotesize}
 \begin{eqnarray*}
R_{n-3}& = &(41-8n, 34-8n,\q^{n-8}, -43+8n,-40+8n, -35+8n, -32+8n,33-8n,\\
&& 42-8n ),\\
R_{n-2}& = &(7-8n, 4-8n, -5+8n, -14+8n, -13+8n, -6+8n,\q^{n-8},15-8n, 12-8n),\\
R_{n-1}&  =& (17-8n, 26-8n, 25-8n, 18-8n,\q^{n-8},-27+8n, -24+8n,-19+8n,\\
&& -16+8n ),\\
R_{n} & =  & ( -1, -4, -9, -12, 11,2,3,10,\q^{n-8}),
\end{eqnarray*}

\end{footnotesize}
\noindent where $i=0,\ldots,\frac{n-10}{2}$ and $j=0,\ldots,\frac{n-8}{2}$. Note that each row has exactly $8$ filled cells.
It is easy to see that $\nn{R_{3}\cup R_{n}}=\{1,\ldots,16\}$,
$\nn{R_{1}\cup R_{n-2}}=\{8n-15,\ldots,8n\}$, and
$\nn{R_{2h}\cup R_{2h+3}}=\{1+16h,\ldots,16+16h\}$,  for all
$h=1,\ldots,\frac{n-4}{2}$.
Hence, conditions (a) and (b) of Definition \ref{def:H} hold. Now, we list the
partial sums for each row. We have
\begin{footnotesize}
\begin{eqnarray*}
\S(R_{1}) & = & (-3+8n, -3+16n, -2+8n, 8, 17-8n, 19-16n, 8-8n, 0 ),\\
\S(R_{2+2i}) & =  & (-17-16i, -37-32i, -62-48i, -90-64i, -63-48i, -45-32i, \\
&&-26-16i, 0), \\
\S(R_{3+2j}) & = & (5+16j, 13+32j, 26+48j, 42+64j, 27+48j, 21+32j, 14+16j, 0 ),\\
\S(R_{n-6}) & = & ( -45+8n, -83+16n, -36+8n, 8, 47-8n, 83-16n, 46-8n, 0  ),\\
\S(R_{n-4}) & = & (-21+8n, -51+16n, -80+24n, -102+32n, -71+24n, -43+16n, \\
&& -20+8n, 0 ),\\
\S(R_{n-3}) & = & (41-8n, 75-16n, 32-8n, -8, -43+8n, -75+16n, -42+8n, 0  ),\\
\S(R_{n-2}) & = & (7-8n, 11-16n, 6-8n, -8,-21+8n, -27+16n, -12+8n, 0),\\
\S(R_{n-1}) & = & ( 17-8n, 43-16n, 68-24n, 86-32n, 59-24n, 35-16n, 16-8n, 0  ),\\
\S(R_{n}) & =  & (-1, -5, -14, -26,-15, -13, -10, 0).
\end{eqnarray*}

\end{footnotesize}
\noindent It is not so hard to see that the elements of each $\S(R_t)$ are pairwise distinct both modulo $16n+1$ and
modulo
$16n+2$ for any even integer $n$.
From the definition of $H$ we obtain the following expression of the columns:

\begin{footnotesize}
\begin{eqnarray*}
C_1 & = & (  -3+8n,\q,5,\q^{n-10}, -45+8n,\q, -21+8n, 41-8n, 7-8n, 17-8n, -1)^T,\\
C_2 & = &( 8n,\q, 8,\q^{n-10}, -38+8n,\q, -30+8n, 34-8n, 4-8n, 26-8n, -4)^T, \\
C_3 & = & ( 1-8n, -17,13,\q,21,\q^{n-10}, -29+8n,\q, -5+8n, 25-8n, -9)^T,\\
C_4 & = & (10-8n, -20, 16,\q, 24,\q^{n-10}, -22+8n,\q, -14+8n, 18-8n, -12)^T,\\
C_5 & = & ( 9-8n, -25, -15, -33, 29,\q, 37,\q^{n-10}, -13+8n,\q, 11 )^T,\\
C_6 & = & (2-8n,-28,-6,-36,32,\q,40,\q^{n-10}, -6+8n,\q, 2 )^T,\\
C_{7} & = & (\q,27,-7,-41,-31,-49,45,\q,53, \q^{n-10}, 3)^T,\\
C_{8} & = & (\q,18,-14,-44,-22,-52,48,\q,56, \q^{n-10}, 10)^T,\\
C_{9+2i} & = & (\q^{1+2i}, 19+16i,\q, 43+16i, -23-16i, -57-16i, -47-16i, -65-16i,\\
&& 61+16i,\q, 69+16i, \q^{n-11-2i} )^T,\\
C_{10+2i} & = & (\q^{1+2i}, 26+16i,\q, 34+16i, -30-16i, -60-16i, -38-16i, -68-16i,\\
&&64+16i,\q, 72+16i, \q^{n-11-2i})^T,\\
C_{n-1} & = & (-11+8n,\q^{n-10}, -61+8n,\q, -37+8n, 57-8n, 23-8n, 33-8n, 15-8n,\\
&&-19+8n,  \q)^T,\\
C_{n} & = & (-8+8n,\q^{n-10}, -54+8n,\q, -46+8n, 50-8n, 20-8n, 42-8n, 12-8n,\\
&& -16+8n, \q)^T,
\end{eqnarray*}

\end{footnotesize}
\noindent where $i=0,\ldots,\frac{n-12}{2}$.
Each column  has $8$ filled cells, hence $H$ satisfies also condition (c).
The partial sums for the columns are the following:

\begin{footnotesize}
\begin{eqnarray*}
\S(C_{1}) & =  & ( -3+8n, 2+8n, -43+16n, -64+24n, -23+16n, -16+8n, 1, 0  ), \\
\S(C_{2}) & =  & ( 8n, 8+8n, -30+16n, -60+24n, -26+16n, -22+8n, 4, 0  ), \\
\S(C_{3}) & =  & ( 1-8n, -16-8n, -3-8n, 18-8n, -11, -16+8n, 9, 0  ), \\
\S(C_{4}) & =  & (  10-8n, -10-8n, 6-8n, 30-8n, 8, -6+8n, 12, 0  ), \\
\S(C_{5}) & =  & ( 9-8n, -16-8n, -31-8n, -64-8n, -35-8n, 2-8n, -11, 0 ), \\
\S(C_{6}) & =  & ( 2-8n, -26-8n, -32-8n, -68-8n, -36-8n, 4-8n, -2, 0   ), \\
\S(C_{7}) & =  & ( 27, 20, -21, -52, -101, -56, -3, 0  ), \\
\S(C_{8}) & =  & ( 18, 4, -40, -62, -114, -66, -10, 0  ), \\
\S(C_{9+2i}) & =  & ( 19+16i, 62+32i, 39+16i, -18, -65-16i, -130-32i, -69-16i, 0 ), \\
\S(C_{10+2i}) & = & (26+16i, 60+32i, 30+16i, -30, -68-16i, -136-32i, -72-16i, 0),\\
\S(C_{n-1}) & =  & ( -11+8n, -72+16n, -109+24n, -52+16n, -29+8n, 4, 19-8n, 0 ), \\
\S(C_{n}) & = & (-8+8n, -62+16n, -108+24n, -58+16n, -38+8n, 4, 16-8n, 0 ).
\end{eqnarray*}

\end{footnotesize}
\noindent Since each column sums to $0$, also condition (d) is satisfied. Hence $H$ is an integer Heffter array.
Then, again by a direct check, one can see that the elements of each $\S(C_t)$ are always pairwise distinct modulo $16n+1$ for any even $n$.
As, by hypothesis,  $n\equiv 0$ or $2\pmod{6}$, the partial sums are distinct also modulo $16n+2$. Thus, for these values of $n$,
$H$ is an $\SH^*(n;8)$.

\noindent \underline{Case 2.}  $n\equiv 4\pmod{6}$.
Let $H$ be the $n\times n$ array whose rows $R_t$ are the following:
\begin{footnotesize}
\begin{eqnarray*}
R_{1}& = &( -3+8n, 8n, 2-8n, 9-8n, 10-8n, 1-8n,\q^{n-8},-11+8n, -8+8n),\\
R_{2+2i} & =  & ( \q^{2+2i}, -17-16i, -20-16i, -25-16i, -28-16i, 26+16i, 19+16i, 18+16i,\\
&& 27+16i, \q^{n-10-2i} ), \\
R_{3+2j} & =  & (\q^{2j}, 5+16j, 8+16j, 13+16j, 16+16j, -14-16j,-7-16j,-6-16j,\\
&& -15-16j,\q^{n-8-2j}),\\
R_{n-6}& =& ( -46+8n, -37+8n,\q^{n-8},47-8n, 44-8n, 39-8n, 36-8n, -38+8n,\\
&&-45+8n),\\
R_{n-3}& = &(42-8n, 33-8n,\q^{n-8}, -43+8n, -40+8n, -35+8n, -32+8n, 34-8n,\\
&&41-8n ),\\
R_{n-2}& = &(7-8n, 4-8n, -6+8n, -13+8n, -14+8n, -5+8n,\q^{n-8},15-8n, 12-8n),\\
R_{n-1}&  =& (18-8n, 25-8n, 26-8n, 17-8n,\q^{n-8},-27+8n, -24+8n, -19+8n,\\
&&-16+8n ),\\
R_{n} & =  & (-1, -4, -9, -12, 10,3,2,11,\q^{n-8}),
\end{eqnarray*}

\end{footnotesize}
\noindent where $i=0,\ldots,\frac{n-10}{2}$ and $j=0,\ldots,\frac{n-8}{2}$. Firstly, we note that each row of $H$ has $8$ filled cells.
Then, one can check that $\left\|R_{3}\cup R_{n}\right\|=\{1,\ldots,16\}$,
$\left\|R_{1}\cup R_{n-2}\right\|=\{8n-15,\ldots,8n\}$, and $\left\|R_{2h}\cup R_{2h+3}\right\|=\{1+16h,\ldots,16+16h\}$,  for all
$h=1,\ldots,\frac{n-4}{2}$. Hence, $H$ satisfies conditions (a) and (b) of Definition \ref{def:H}. Now, we list the
partial sums for each row. We have
\begin{footnotesize}
\begin{eqnarray*}
\S(R_{1})& = &( -3+8n, -3+16n, -1+8n, 8, 18-8n, 19-16n, 8-8n, 0 ),\\
\S(R_{2+2i})& = &( -17-16i, -37-32i, -62-48i, -90-64i, -64-48i, -45-32i, \\
&&-27-16i, 0 ),\\
\S(R_{3+2j}) & = & (5+16j, 13+32j, 26+48j, 42+64j, 27+48j, 21+32j, 14+16j, 0 ),\\
\S(R_{n-6})& = &( -46+8n, -83+16n, -36+8n, 8, 47-8n, 83-16n, 45-8n, 0 ),\\
\S(R_{n-4})& = &( -22+8n, -51+16n, -81+24n, -102+32n, -71+24n, -43+16n, \\
&&-20+8n, 0 ),\\
\S(R_{n-3})& = &( 42-8n, 75-16n, 32-8n, -8, -43+8n, -75+16n, -41+8n, 0 ),\\
\end{eqnarray*}
\end{footnotesize}
\begin{footnotesize}
\begin{eqnarray*}
\S(R_{n-2})& = &( 7-8n, 11-16n, 5-8n, -8, -22+8n, -27+16n, -12+8n, 0 ),\\
\S(R_{n-1})& = &( 18-8n, 43-16n, 69-24n, 86-32n, 59-24n, 35-16n, 16-8n, 0 ),\\
\S(R_{n})& = &( -1, -5, -14, -26, -16, -13, -11, 0 ).
\end{eqnarray*}

\end{footnotesize}
\noindent Note that each row sums to $0$.
It is not hard  to check that the elements of each $\S(R_t)$ are pairwise distinct modulo $16n+1$ and modulo $16n+2$,
for any even $n$.
From the definition of $H$ we obtain the following expression of the columns:
\begin{footnotesize}
\begin{eqnarray*}
C_1 & = & (  -3+8n,\q,5,\q^{n-10}, -46+8n,\q, -22+8n, 42-8n, 7-8n, 18-8n, -1)^T,\\
C_2 & = &( 8n,\q, 8,\q^{n-10}, -37+8n,\q, -29+8n, 33-8n, 4-8n, 25-8n, -4)^T, \\
C_3 & = & ( 2-8n, -17,13,\q,21,\q^{n-10}, -30+8n,\q, -6+8n, 26-8n, -9)^T,\\
C_4 & = & (9-8n, -20, 16,\q, 24,\q^{n-10}, -21+8n,\q, -13+8n, 17-8n, -12)^T,\\
C_5 & = & ( 10-8n, -25, -14, -33, 29,\q, 37,\q^{n-10}, -14+8n,\q, 10 )^T,\\
C_6 & = & (1-8n,-28,-7,-36,32,\q,40,\q^{n-10}, -5+8n,\q, 3 )^T,\\
C_{7} & = & (\q,26,-6,-41,-30,-49,45,\q,53, \q^{n-10}, 2)^T,\\
C_{8} & = & (\q,19,-15,-44,-23,-52,48,\q,56, \q^{n-10}, 11)^T,\\
C_{9+2i} & = & (\q^{1+2i}, 18+16i,\q, 42+16i, -22-16i, -57-16i, -46-16i, -65-16i,\\
&& 61+16i,\q, 69+16i, \q^{n-11-2i} )^T,\\
C_{10+2i} & = & ( \q^{1+2i}, 27+16i,\q, 35+16i, -31-16i, -60-16i, -39-16i, -68-16i, \\
&&  64+16i,\q, 72+16i, \q^{n-11-2i})^T,\\
C_{n-1} & = & ( -11+8n,\q^{n-10}, -62+8n,\q, -38+8n, 58-8n, 23-8n, 34-8n, \\
&& 15-8n, -19+8n , \q)^T,\\
C_{n} & = & (-8+8n,\q^{n-10}, -53+8n,\q, -45+8n, 49-8n, 20-8n, 41-8n, \\
&& 12-8n, -16+8n ,  \q)^T,
\end{eqnarray*}

\end{footnotesize}
\noindent  where $i=0,\ldots,\frac{n-12}{2}$.
Since also each column has exactly $8$ filled cells, then condition (c) is satisfied.
Now we list the partial sums for the columns:

\begin{footnotesize}
\begin{eqnarray*}
\S(C_{1}) & =  & ( -3+8n, 2+8n, -44+16n, -66+24n, -24+16n, -17+8n, 1, 0 ),\\
\S(C_{2}) & =  & ( 8n, 8+8n, -29+16n, -58+24n, -25+16n, -21+8n, 4, 0 ), \\
\S(C_{3}) & =  & ( 2-8n, -15-8n, -2-8n, 19-8n, -11, -17+8n, 9, 0 ), \\
\S(C_{4}) & =  & (9-8n, -11-8n, 5-8n, 29-8n, 8, -5+8n, 12, 0 ), \\
\S(C_{5}) & =  & ( 10-8n, -15-8n, -29-8n, -62-8n, -33-8n, 4-8n, -10, 0 ),\\
\S(C_{6}) & =  & ( 1-8n, -27-8n, -34-8n, -70-8n, -38-8n, 2-8n, -3, 0 ),\\
\S(C_{7}) & =  & (  26, 20, -21, -51, -100, -55, -2, 0  ), \\
\S(C_{8}) & =  & (  19, 4, -40, -63, -115, -67, -11, 0 ), \\
\S(C_{9+2i}) & =  & (18+16i, 60+32i, 38+16i, -19, -65-16i, -130-32i, -69-16i, 0  ), \\
\S(C_{10+2i}) & = & ( 27+16i, 62+32i, 31+16i, -29, -68-16i, -136-32i, -72-16i, 0 ),\\
\S(C_{n-1}) & =  & (  -11+8n, -73+16n, -111+24n, -53+16n, -30+8n, 4, 19-8n, 0 ), \\
\S(C_{n}) & = & (  -8+8n, -61+16n, -106+24n, -57+16n, -37+8n, 4, 16-8n, 0 ).
\end{eqnarray*}

\end{footnotesize}
\noindent Each column sums to $0$, hence $H$ satisfies also condition (d).
So, $H$ is an integer Heffter array.
Finally, again by a direct check, one can see that, since $n\equiv 4 \pmod{6}$,
the elements of each $\S(C_t)$ are pairwise distinct both modulo $16n+1$ and modulo $16n+2$.
Thus, $H$ is an $\SH^*(n;8)$ for any $n\equiv 4 \pmod 6$.
\end{proof}

We point out that the $\SH^*(10;8)$  given in Example \ref{10.8} has been obtained following the proof of Proposition \ref{0,2}.

\begin{prop}\label{1,3}
Let $n\geq 9$ be odd. Then, there exists an $\SH^*(n;8)$.
\end{prop}

\begin{proof}
If $n=9,11,13,15,17,19$ an $\SH^*(n;8)$ can be found in \cite{web}.
Let now $n\geq 21$ and let  $a=n-9$, obviously $a$ is an even integer and $a\geq12$.
Let $H$ be the $(a+9)\times (a+9)$ array whose first $a$ rows are the ones of the $a\times a$ array constructed in
 Proposition \ref{0,2} with nine empty cells at the end and the last nine rows are the following:
\begin{footnotesize}
\begin{eqnarray*}
R_{a+1}& =& (\q^a,3+8a,-61-8a,-20-8a,-19-8a,68+8a,\q,44+8a,36+8a,-51-8a),\\
R_{a+2}& =&(\q^a,30+8a,\q,38+8a,-46-8a,-23-8a,-13-8a,71+8a,-63-8a,6+8a),\\
R_{a+3}& =&(\q^{a+1},43+8a,2+8a,-10-8a,-50-8a,67+8a,35+8a,-27-8a,-60-8a),\\
R_{a+4}& =&(\q^a,-48-8a,-16-8a,65+8a,\q,41+8a,-58-8a,-26-8a,9+8a,33+8a),\\
R_{a+5}& =&(\q^a,-12-8a,70+8a,29+8a,37+8a,5+8a,-22-8a,-53-8a,-54-8a,\q),\\
R_{a+6}& =&(\q^a,66+8a,34+8a,\q,1+8a,-59-8a,-49-8a,-17-8a,-18-8a,42+8a),\\
R_{a+7}& =&(\q^a,-21-8a,-52-8a,-11-8a,28+8a,\q,4+8a,-62-8a,45+8a,69+8a),\\
R_{a+8}& =&(\q^a,39+8a,7+8a,-56-8a,-55-8a,-14-8a,31+8a,\q,72+8a,-24-8a),\\
R_{a+9}& =&(\q^a,-57-8a,-25-8a,-47-8a,64+8a,32+8a,40+8a,8+8a,\q,-15-8a).
\end{eqnarray*}

\end{footnotesize}
\noindent Note that these rows have exactly $8$ filled cells.
Also $\nn{\cup_{h=1}^{9} R_{a+h}}=\{8a+1,\ldots, 8a+72\}$.
Hence $H$ satisfies conditions (a) and (b) of Definition \ref{def:H}.
Clearly, the lists of the partial sums of the first $a$ rows of $H$ are the same
written in the proof of Proposition \ref{0,2}.
So, we list only the partial sums for the last nine rows:
\begin{footnotesize}
\begin{eqnarray*}
\S(R_{a+1}) & = & ( 3+8a, -58, -78-8a, -97-16a, -29-8a, 15, 51+8a, 0),\\
\S(R_{a+2}) & = & ( 30+8a, 68+16a, 22+8a, -1, -14-8a, 57, -6-8a, 0),\\
\S(R_{a+3}) & = & ( 43+8a, 45+16a, 35+8a, -15, 52+8a, 87+16a, 60+8a, 0),\\
\S(R_{a+4}) & = & ( -48-8a, -64-16a, 1-8a, 42, -16-8a,-42-16a, -33-8a,0),\\
\S(R_{a+5}) & = & ( -12-8a, 58, 87+8a, 124+16a, 129+24a, 107+16a, 54+8a, 0),\\
\S(R_{a+6}) & = & ( 66+8a, 100+16a, 101+24a, 42+16a, -7+8a, -24, -42-8a, 0),\\
\S(R_{a+7}) & = & ( -21-8a, -73-16a, -84-24a, -56-16a, -52-8a, -114-16a,\\
&&-8a-69, 0),\\
\S(R_{a+8}) & = & ( 39+8a, 46+16a, -10+8a, -65, -79-8a, -48, 24+8a, 0),\\
\S(R_{a+9}) & = & ( -57-8a, -82-16a, -129-24a, -65-16a, -33-8a, 7, 15+8a, 0 ).
\end{eqnarray*}

\end{footnotesize}
\noindent Note that each row sums to $0$. By a long and direct verification one can see that
the elements of each $\S(R_t)$, $1\leq t\leq a+9$, are pairwise distinct both modulo
$16(a+9)+1$ and modulo $16(a+9)+2$.

Now, since the first $a$ cells of each row $R_{a+h}$, $1\leq h \leq 9$, are empty,
the first $a$ columns of $H$ are the ones of the $a\times a$ array defined by Proposition \ref{0,2}
with nine empty cells at the end. Also,
 the last nine columns are  the following:
\begin{footnotesize}
\begin{eqnarray*}
C_{a+1}& =& (\q^a,3+8a, 30+8a, \q, -48-8a, -12-8a, 66+8a, -21-8a, 39+8a, -57-8a )^T,\\
C_{a+2}& =&(\q^a,-61-8a, \q, 43+8a, -16-8a, 70+8a, 34+8a, -52-8a, 7+8a, -25-8a)^T,\\
C_{a+3}& =&(\q^a,-20-8a, 38+8a, 2+8a, 65+8a, 29+8a, \q, -11-8a, -56-8a, -47-8a)^T,\\
C_{a+4}& =&(\q^a,-19-8a, -46-8a, -10-8a, \q, 37+8a, 1+8a, 28+8a, -55-8a, 64+8a)^T,\\
C_{a+5}& =&(\q^a,68+8a, -23-8a, -50-8a, 41+8a, 5+8a, -59-8a, \q, -14-8a, 32+8a )^T,\\
C_{a+6}& =&(\q^{a+1}, -13-8a, 67+8a, -58-8a, -22-8a, -49-8a, 4+8a, 31+8a, 40+8a )^T,\\
C_{a+7}& =&(\q^a,44+8a, 71+8a, 35+8a, -26-8a, -53-8a, -17-8a, -62-8a, \q, 8+8a)^T,\\
C_{a+8}& =&(\q^a,36+8a, -63-8a, -27-8a, 9+8a, -54-8a, -18-8a, 45+8a, 72+8a, \q)^T,\\
C_{a+9}& =&(\q^a,-51-8a, 6+8a, -60-8a, 33+8a, \q, 42+8a, 69+8a, -24-8a, -15-8a )^T.
\end{eqnarray*}

\end{footnotesize}
\noindent Since also these columns have exactly $8$ filled cells, $H$ satisfies  condition (c).
Obviously, the lists of the partial sums of the first $a$ columns of $H$ are the same
written in the proof of Proposition \ref{0,2}.
So, as done for the rows, we list only the partial sums for the last nine columns:
\begin{footnotesize}
\begin{eqnarray*}
\S(C_{a+1}) & = & ( 3+8a, 33+16a, -15+8a, -27, 39+8a, 18, 57+8a, 0),\\
\S(C_{a+2}) & = & ( -61-8a, -18, -34-8a, 36, 70+8a, 18, 25+8a, 0),\\
\S(C_{a+3}) & = & ( -20-8a, 18, 20+8a, 85+16a, 114+24a, 103+16a, 47+8a, 0),\\
\S(C_{a+4}) & = & ( -19-8a, -65-16a, -75-24a, -38-16a, -37-8a, -9, -64-8a, 0 ),\\
\S(C_{a+5}) & = & ( 68+8a, 45, -5-8a, 36, 41+8a, -18, -32-8a, 0 ),\\
\S(C_{a+6}) & = & ( -13-8a, 54, -4-8a, -26-16a, -75-24a, -71-16a, -40-8a, 0),\\
\S(C_{a+7}) & = & ( 44+8a, 115+16a, 150+24a, 124+16a, 71+8a, 54, -8-8a, 0 ),\\
\S(C_{a+8}) & = & ( 36+8a, -27, -54-8a, -45, -99-8a, -117-16a, -72-8a, 0),\\
\S(C_{a+9}) & = & ( -51-8a, -45, -105-8a, -72,-30+8a, 39+16a, 15+8a, 0).\\
\end{eqnarray*}

\end{footnotesize}
\noindent Each column sums to $0$, so also condition (d) is satisfied.
Hence $H$ is an integer Heffter array.
Finally, again by a direct check, one can see that the elements of each $\S(C_t)$, $1\leq t\leq a+9$,
 are pairwise distinct both modulo $16(a+9)+1$ and modulo $16(a+9)+2$. Thus, $H$ is an $\SH^*(n;8)$.
\end{proof}

\begin{prop} \label{SH9_even}
Let  $n\equiv 0 \pmod 4$ and $n\geq 12$. Then, there  exists an $\SH^*(n; 9)$.
\end{prop}

\begin{proof}
Let $n=4a$ and let $H$ be the $n\times n$ array whose rows $R_t$ are defined as follows:

\begin{footnotesize}
\begin{eqnarray*}
  R_1 & =  & (-2+ 12a, 2-4a, -2+16a, -4+36a, 2-36a , \q^{2a-5},  -1-12a, 3-12a , \\
  &&\q^{2a-4},   5-36a, -3+36a  ),\\
  R_2 & =&  ( 1+20a, 1-16a, 3-4a, -5+20a, 36a, -2-20a,  \q^{2a-5} , -4+12a, 5-12a,\\
  &&\q^{2a-4}, 1-36a ), \\
  R_{3+2i} & =& (\q^{2i},  -3-20a-8i, 5+20a+8i, -2-12a-i, 4-4a+2i, -3+16a-i,\\
  &&4+20a+8i, -6-20a-8i  , \q^{2a-5},      -6+12a-4i, 7-12a+4i , \q^{2a-4-2i} ),\\
   R_{4+2j} & = & (\q^{2j+1},  -7-20a-8j, 9+20a+8j, -16a-j, 5-4a+2j, -6+20a-j,\\
   &&8+20a+8j, -10-20a-8j , \q^{2a-5}, -8+ 12a-4j, 9-12a+4j, \q^{2a-5-2j}),\\
   R_{2a} & = & (2-8a, \q^{2a-4},  9-28a, -7+28a, 2-17a, -3+20a, -1+17a, -8+28a, \\
   && 6-28a , \q^{2a-5} , -12a), \\
   R_{2a+1+2j} &= & (\q^{2j}, -1+4a+4j, -4a-4j , \q^{2a-4},  5-28a-8j, -3+28a+8j, 2-15a+j, \\
   && -2+2a-2j, 1+13a+j, -4+28a+8j, 2-28a-8j , \q^{2a-5-2j} ),\\
  R_{2a+2+2j} &= & ( \q^{1+2j},  1+4a+4j, -2-4a-4j ,\q^{2a-4},  1 -28a-8j,1+ 28a+8j,\\
	&& 4-19a+j, -3+2a-2j, 17a+j, 28a+8j, -2-28a-8j,  \q^{2a-6-2j} ),\\
 R_{4a-3} & =&  (18-36a,\q^{2a-5}, -9+ 8a, 8-8a , \q^{2a-4},  21-36a, -19+36a, -14a, 2, \\
 && -1+14a, -20+36a ),\\
 R_{ 4a-2} &=& ( -16+36a, 14-36a, \q^{2a-5}, -7+8a, 6-8a, \q^{2a-4}, 17 -36a, -15+36a, \\
 && 2-18a, 1-2a, -2+20a),\\
 R_{4a-1} &= & ( -8a, -12+36a, 10-36a , \q^{2a-5}, -5+8a, 20a ,\q^{2a-4},  13-36a, -11+36a, \\
 &&1-12a, 4-8a ),\\
R_{ 4a} & =&  ( 1, -4+20a, -8+ 36a, 6-36a ,\q^{2a-5},  1-20a, 3-8a ,\q^{2a-4},  9-36a,\\
&&-7+36a, -1+8a ),
\end{eqnarray*}

\end{footnotesize}
\noindent where $i=0,\ldots,a-2$ and $j=0,\ldots,a-3$; hence,  each row has $9$ filled cells.
Since

\begin{footnotesize}
\begin{gather*}
\begin{array}{rcl}
\nn{R_1\cup R_2 \cup \cup_{t=4a-3}^{4a} R_t} & =&
\{1,2,2a-1,4a-3, 4a-2 \}\cup\{8a-9,\ldots,  8a-3 \}\cup\\
&& \{8a-1,8a \}\cup \{12a-5,\ldots,12a-1 \}\cup \{12a+1,14a-1,\\
&& 14a,16a-2,16a-1, 18a-2,20a-5, 20a-4\}\cup\\
&&\{20a-2,\ldots,20a+2 \}\cup \{36a-21,\ldots,36a\},
\end{array}\\
\begin{array}{rcl}
\cup_{j=0}^{a-3}\nn{R_{3+2j}\cup R_{4+2j}} & =& \{2a+1,\ldots,4a-4\}\cup\{8a+3,\ldots,12a-6\}\cup
\{ 12a+2,\ldots,\\
&&13a-1\}\cup \{15a,\ldots,16a-3 \} \cup \{16a,\ldots,17a-3 \}\cup \{ 19a-3,\\
&&\ldots,20a-6\} \cup \{20a+3,\ldots,28a-14\},
\end{array}\\
\begin{array}{rcl}
\nn{R_{2a-1}\cup R_{2a} } & =& \{2a,8a-2,8a+1, 8a+2,12a,13a,15a-1,17a-2,17a-1,20a-3\}\cup\\ && \{28a-13,\ldots,28a-6\},
\end{array}\\
\begin{array}{rcl}
\cup_{j=0}^{a-3}\nn{R_{2a+1+2j}\cup R_{2a+2+2j}} & =& \{3,\ldots,2a-2
\}\cup\{4a-1,\ldots,8a-10\}\cup \{13a+1,\ldots,\\
&&14a-2 \}\cup
\{14a+1,\ldots,15a-2 \}\cup\{17a,\ldots,18a-3\}\cup\\
&&\{18a-1,\ldots, 19a-4\}\cup \{28a-5,\ldots,36a-22\},
\end{array}
\end{gather*}

\end{footnotesize}
\noindent  $H$ satisfies conditions (a) and (b) of Definition \ref{def:H}. Now, we list the partial
sums for each row. We have
\begin{footnotesize}
\begin{eqnarray*}
 \S(R_1) & = & ( -2+12a, 8a, -2+24a, -6+60a, -4+24a, -5+12a, -2, 3-36a, 0 ),\\
 \S(R_2) & = & ( 1+20a, 2+4a, 5, 20a, 56a, -2+36a, -6+48a, -1+36a, 0 ),\\
\S(R_{3+2i}) & =&   ( -3-20a -8i, 2, -12a -i, 4-16a+ i, 1, 5+20a+ 8i, -1,\\
&&-7+12a -4i, 0 ),\\
\S(R_{4+2j}) & =&   (-7-20a -8j, 2, 2-16a -j,   7-20a+ j, 1, 9+20a +8j, -1,  \\
&& -9+12a -4j, 0 ),\\
\S(R_{2a}) & =&   (2 -8a,  11 -36a, 4 -8a,  6 -25a, 3 -5a, 2+ 12a,  -6+ 40a, 12a, 0 ),\\
  \S(R_{2a+1+2j})& =&  ( -1+4a+4j, -1, 4-28a-8j , 1,   3-15a+j,1-13a -j, 2,\\
&& -2+28a+8j, 0 ), \\
\S(R_{2a+2+2j}) & =&   (1+4a+ 4j, -1,  -28a-8j, 1, 5-19a+ j,  2-17a -j, 2, 2+28a+ 8j, 0 ), \\
  \S(R_{4a-3}) & =&   ( 18-36a, 9-28a, 17 -36a,  38-72a,19 -36a, 19-50a,  21-50a,\\
&&20-36a, 0 ), \\
\S(R_{4a-2}) &= &  ( -16+36a, -2, -9+8a, -3,14 -36a, -1,1 -18a, 2-20a, 0 ),\\
\S(R_{4a-1})& = & ( -8a, -12+28a, -2-8a, -7, -7+20a, 6 -16a, -5+20a, -4+8a, 0 ), \\
\S(R_{4a}) & =&   ( 1,-3+ 20a, -11+56a, -5+20a, -4,-1 -8a, 8-44a,1 -8a, 0 ) .
\end{eqnarray*}

\end{footnotesize}
\noindent By a long direct calculation, the reader can check that the elements of each $\S(R_t)$ are pairwise distinct
both modulo $72a+1$ and modulo $72a+2$
and in particular each row sums to $0$. From the definition of $H$ we obtain the  following expression of the columns:
\begin{footnotesize}
\begin{eqnarray*}
 C_1 & =& ( -2+12a, 1+20a, -3-20a,\q^{2a-4} ,2-8a, -1+4a, \q^{2a-5} , 18-36a,\\
 &&-16+36a, -8a, 1 )^T,\\
C_2 & =& ( 2-4a, 1-16a, 5+20a, -7-20a ,\q^{2a-4}, -4a, 1+4a,  \q^{2a-5},14 -36a,\\
&& -12+ 36a, -4+20a )^T,\\
C_3 & =& ( -2+16a, 3-4a, -2-12a, 9+20a, -11-20a ,\q^{2a-4} ,-2-4a,  3+4a,\\
&&\q^{2a-5}, 10-36a, -8+36a )^T,\\
C_4 & =& ( -4+36a, -5+20a, 4-4a, -16a, 13+20a, -15-20a ,\q^{2a-4},-4-4a,\\
&& 5+4a,\q^{2a-5},6-36a)^T,\\
C_5 & =& ( 2-36a, 36a, -3+16a, 5-4a, -3-12a, 17+20a, -19-20a ,\q^{2a-4},\\
&&-6-4a,7+4a,\q^{2a-5})^T,\\
C_{6+2h} & =& (\q^{1+2h},-2-20a-8h, 4+20a+8h, -6+20a-h, 6-4a+2h, -1-16a-h,\\
&& 21+20a+8h, -23-20a-8h  ,     \q^{2a-4}, -8-4a-4h, 9+4a+4h  , \q^{2a-6-2h})^T,\\
\end{eqnarray*}
  \end{footnotesize}
  \begin{footnotesize}
  \begin{eqnarray*}
C_{7+2h} & =& (\q^{2+2h}, -6-20a-8h, 8+20a+8h, -4+16a-h, 7-4a+2h, -4-12a-h, \\
&&25+20a+8h, -27-20a-8h , \q^{2a-4}, -10-4a-4h, 11+4a+4h , \\
&&  \q^{2a-7-2h}  )^T,\\
C_{2a} & = & ( \q^{2a-5},  22-28a, -20+28a, -3+19a, -2a, 2-17a, -3+28a, 1-28a , \\
&& \q^{2a-4}, 20a, 1-20a )^T,\\
C_{2a+1} & =& (-1-12a,\q^{2a-5}, 18-28a, -16+28a, -1+15a, -3+20a, 2-15a,\\
&& 1+28a, -3-28a , \q^{2a-4}, 3-8a)^T,\\
C_{2a+2+2i} &= &(  \q^{2i}, 3-12a+4i, -4+12a-4i ,\q^{2a-5}, 14-28a-8i, -12+28a+8i, \\
&&-1+17a+i, -2+2a-2i, 4-19a+i, 5+28a+8i, -7-28a-8i,\\
&& \q^{2a-4-2i} )^T,\\
C_{2a+3+2j} & =& (\q^{1+2j}, 5-12a+4j, -6+12a-4j, \q^{ 2a-5},      10-28a-8j, -8+28a+8j,\\
&& 1+13a+j, -3+2a-2j, 3-15a+j, 9+28a+8j, -11-28a-8j ,\\
&&\q^{ 2a-5-2j} )^T,\\
C_{4a-1} &=& (5-36a ,\q^{2a-4} , -3-8a, 2+8a, \q^{2a-5},  26-36a, -24+36a, -1+14a,\\
&& 1-2a, 1-12a, -7+36a)^T,\\
C_{4a} &= &( -3+36a, 1-36a, \q^{2a-4}, -1-8a, -12a, \q^{2a-5}, 22-36a, -20+36a,\\
&&-2+20a, 4-8a, -1+8a )^T,
 \end{eqnarray*}

\end{footnotesize}
\noindent where $h=0,\ldots,a-4$, $i=0,\ldots,a-2$ and $j=0,\ldots,a-3$. Note that each column has $9$ filled cells and so condition (c) holds.
We have:
\begin{footnotesize}
\begin{eqnarray*}
\S(C_1) & =& (  -2+  12a,   -1+  32a,  -4+  12a,       -2+ 4a, -3+    8a,  15    -28a,  -1+  8a,     -1,   0 ),\\
\S(C_2) &=& (  2 -4a, 3-20a, 8, 1-20a,1 -24a, 2-20a,16-56a,4 -20a,  0 ),\\
  \S(C_3) & =& (  -2+16a, 1+ 12a,-1, 8+ 20a,-3, -5 -4a,-2, 8-36a,   0 ),\\
  \S(C_4) &=&   (  -4+36a, -9+ 56a, -5+ 52a, -5+ 36a, 8+ 56a, -7+ 36a, -11+32a,\\
&&-6+ 36a, 0 ),\\
\S(C_5) &=&   (2 -36a, 2,  -1+16a, 4+ 12a, 1, 18+20a, -1,-7  -4a,   0 ),\\
\S(C_{6+2h}) & =&   ( -2-20a  -8h ,  2,  -4+20a   -h,   2+16a+h, 1, 22+20a+  8h,-1,\\
&&-9-4a   -4h, 0 ),\\
\S(C_{7+2h}) & =&   ( -6-20a  -8h, 2,   -2+16a  -h, 5+12a+h,  1,   26+20a+  8h, -1,  \\
&& -11-4a -4h,  0 ),\\
\S(C_{2a}) & =&   (22-28a, 2, -1+ 19a,  -1+17a, 1,  -2+28a,-1,  -1+20a,     0 ),\\
\S(C_{2a+1})& =&   ( -1-12a, 17-40a,1 -12a,     3a,-3+  23a, -1 + 8a,    36a, -3+  8a, 0 ),\\ 
\S(C_{2a+2+2i})&= &  (   3-12a+ 4i, -1,  13-28a  -8i, 1, 17a+ i,   -2+19a   -i, 2,   7+28a+ 8i, 0 ),\\
\S(C_{2a+3+2j}) & = &   (  5-12a+ 4j,-1,  9-28a  -8j,  1,  2+13a+ j,   -1+15a  -j, 2, \\
&&11+28a+  8j,0 ),\\
\S(C_{4a-1})  & =& (5 -36a, 2-44a, 4-36a,  30-72a,6 -36a,  5-22a, 6-24a, 7-36a, 0 ),\\
\S(C_{4a}) & =&   (-3+  36a,-2,  -3 -8a, -3-20a,  19-56a,-1 -20a,-3, 1  -8a,       0 ) .
\end{eqnarray*}

\end{footnotesize}
\noindent Since every column sums to $0$, condition (d) is satisfied and so $H$ is an $\H(4a; 9)$.
Also in this case, the elements of each $\S(C_t)$ are pairwise distinct both modulo $72a+1$ and modulo $72a+2$.
We conclude that $H$ is an $\SH^*(4a;9)$.
\end{proof}

\begin{ex}
By the proof of  Proposition \ref{SH9_even}, we obtain the following $\SH^*(12;9)$:
\begin{footnotesize}
$$\begin{array}{|r|r|r|r|r|r|r|r|r|r|r|r|}\hline
34 & -10 & 46 & 104 & -106 &  & -37 & -33 &  &  & -103 & 105 \\\hline
61 & -47 & -9 & 55 & 108 & -62 &  & 32 & -31 &  &  & -107 \\\hline
-63 & 65 & -38 & -8 & 45 & 64 & -66 &  & 30 & -29 &  &  \\\hline
 & -67 & 69 & -48 & -7 & 54 & 68 & -70 &  & 28 & -27 &  \\\hline
 &  & -71 & 73 & -39 & -6 & 44 & 72 & -74 &  & 26 & -25 \\\hline
-22 &  &  & -75 & 77 & -49 & 57 & 50 & 76 & -78 &  & -36 \\\hline
11 & -12 &  &  & -79 & 81 & -43 & 4 & 40 & 80 & -82 &  \\\hline
 & 13 & -14 &  &  & -83 & 85 & -53 & 3 & 51 & 84 & -86 \\\hline
-90 &  & 15 & -16 &  &  & -87 & 89 & -42 & 2 & 41 & 88 \\\hline
92 & -94 &  & 17 & -18 &  &  & -91 & 93 & -52 & -5 & 58 \\\hline
-24 & 96 & -98 &  & 19 & 60 &  &  & -95 & 97 & -35 & -20 \\\hline
1 & 56 & 100 & -102 &  & -59 & -21 &  &  & -99 & 101 & 23 \\\hline
\end{array}$$
\end{footnotesize}
\end{ex}

\begin{prop}\label{SH9_odd}
Let $n\equiv 3 \pmod 4$ and  $n\geq 11$. Then, there exists an $\SH^*(n; 9)$.
\end{prop}

\begin{proof}
An $\SH^*(11;9)$ can be found in \cite{web}. So, we assume $n\geq 15$. We split the proof into two cases.\\
\noindent \underline{Case 1.} Let $n=8a+3$ and  $H$ be the $n\times n$ array whose rows $R_t$ are defined as follows:
\begin{footnotesize}
\begin{eqnarray*}
R_1 & = & (2+8a,23+ 64a, -9-24a, -24-64a,13+ 40a,\q^{8a-6}, -14-40a, -17-48a,\\&& 8+ 16a,18+ 48a
),\\
R_2 & = & ( 16+40a, -1-6a,27+ 72a, -7-20a, -22-56a,12+ 31a,\q^{8a-6}, -12-25a,\\&& -21-56a,
8+20a),\\
R_3 & = & (4+8a,20+ 48a, 4a,25+ 64a,-5 -8a,-26 -64a,13+ 26a, \q^{8a-6}, -12-30a,\\&&-19 -48a ),\\
R_4 & = & (-17-40a, 6+12a,18+ 40a,-1+ 2a,23+ 56a, -7-12a, -24-56a,\\&&14+ 35a,\q^{8a-6}, -12-37a
),\\
R_{5+8i} & =& (\q^{8i}, -12-24a-i, -21-48a-8i, 8+24a-4i, 22+48a+8i, -1-8a+2i,\\&& 27+64a+8i,
-7-24a+4i, -28-64a-8i, 12+32a-i,\q^{8a-6-8i} ),\\
  R_{6+8i} & = & (\q^{1+8i}, -13-33a-i, -19-40a-8i, 6+20a-4i, 20+40a+8i, -6a+2i,\\&& 25+56a+8i,
-5-20a+4i, -26-56a-8i,12+ 39a-i, \q^{8a-7-8i} ),\\
R_{7+8i} & = & (\q^{2+8i},-12-38a+i, -23-48a-8i, 6+8a+4i,24+ 48a+8i,-1+ 4a-2i, \\&& 29+64a+8i, -7-8a-4i,
-30-64a-8i, 14+34a+i,\q^{8a-8-8i} ),\\
R_{8+8j} & = & (\q^{3+8j}, -11-29a+j, -21-40a-8j,8+ 12a+4j,22+ 40a+8j,-2+ 2a-2j,\\&& 27+56a+8j,
-9-12a-4j, -28-56a-8j, 14+27a+j,\q^{8a-9-8j}),\\
R_{9+8j} & = & (\q^{4+8j}, -13-32a-j, -25-48a-8j, 6+24a-4j,26+ 48a+8j, -8a+2j,\\&& 31+ 64a+8j,
-5-24a+4j, -32-64a-8j, 12+40a-j ,\q^{8a-10-8j}),\\
R_{10+8j} & = & (\q^{5+8j}, -13-25a-j, -23-40a-8j, 4+20a-4j, 24+40a+8j, 1-6a+2j,\\&& 29+ 56a+8j,
-3-20a+4j, -30-56a-8j, 11+31a-j,\q^{8a-11-8j} ),\\
R_{11+8j} & = & (\q^{6+8j}, -11-30a+j, -27-48a-8j,8+ 8a+4j, 28+48a+8j, -2+4a-2j,\\&& 33+ 64a+8j, -9-8a-4j,
-34-64a-8j,14+ 26a+j,\q^{8a-12-8j} ),\\
R_{12+8h} & = & (\q^{7+8h}, -11-37a+h, -25-40a-8h, 10+12a+4h, 26+40a+8h, \\
&& -3+2a-2h, 31+56a+8h, -11-12a-4h, -32-56a-8h, 15+35a+h,\\
&& \q^{8a-13-8h}),\\
R_{8a-4} & = & (\q^{8a-9}, -13-36a,-9 -48a,2+ 16a,10+ 48a,-1 -4a,15+ 64a,-3 -16a,\\&& -16 -64a,15+ 40a,\q^3
),\\
\end{eqnarray*}
  \end{footnotesize}
  \begin{footnotesize}
  \begin{eqnarray*}
R_{8a} &= & ( 11+24a,\q^{8a-6}, -10-24a, -13-48a,4+ 16a,14+ 48a, 1,19+ 64a, -6-16a,\\&& -20-64a),\\
R_{8a+1} & =& ( -24-72a, 13+39a,\q^{8a-6}, -12-33a, -17-56a,10+ 20a, 18+56a, -2-6a, \\&& 23+ 72a,
-9-20a),\\
R_{8a+2} & = & (-5-16a, -22-64a,13+ 34a,\q^{8a-6}, -12-26a, -15-48a, 7+16a,16+ 48a,\\&& -3-8a,
21+64a ),\\
R_{8a+3} & =& (25+ 72a, -5-12a, -26-72a, 13+27a, \q^{8a-6},-12-29a, -19-56a, 4+12a,\\&& 20+ 56a, 2a ),
 \end{eqnarray*}

\end{footnotesize}
\noindent where $i=0,\ldots,a-1$, $j=0,\ldots,a-2$ and $h=0,\ldots, a-3$. Note that every row contains exactly $9$ elements. Since

  \begin{footnotesize}
  \begin{gather*}
\begin{array}{rcl}
\nn{\cup_{t=3}^{5} R_{8a-2t} \cup \cup_{i=1}^3 R_{8a-i}}& =&
\{2,2a+1,4a+2,4a+3,4a+4,6a+3, 12a+2, 12a+3,\\
&&16a,16a+1\} \cup \{16a+9,\ldots,16a+14\} \cup\{20a+11,\\
&&20a+12,25a+11,26a+11,28a+12,28a+13,30a+13,\\
&&31a+13,34a+11,34a+12,35a+13,37a+13,38a+13,\\
&&38a+14\}\cup
    \{48a+3,  \ldots,48a+8\}\cup  \{48a+11,48a+12\}\\
		&& \cup     \{56a+13,\ldots,56a+16\} \cup \{64a+9,\ldots, 64a+14\} \cup\\
&&    \{64a+17,64a+18\} \cup  \{72a+19,\ldots,72a+22\},
 \end{array}\\
 \begin{array}{rcl}
\cup_{j=0}^{a-2}\nn{\cup_{i=2}^5 R_{2i+1+8j}} & =& \{2a+2,\ldots,4a-1\} \cup \{6a+4,\ldots,8a+1\}
\cup \{8a+6,\ldots,\\
&&12a+1\} \cup \{20a+13, \ldots,24a+8 \} \cup \{24a+12, \ldots, 25a+10\} \cup\\
&& \{26a+14,\ldots, 27a+12\}\cup\{29a+13,\ldots,30a+11\} \cup \\
&&\{31a+14 ,\ldots, 33a+11\}\cup \{ 34a+14,\ldots,35a+12\}\cup \\
&&\{ 37a+14,\ldots,38a+12\}\cup \{ 39a+14,\ldots,40a+12\}\cup \\
&&\{ 48a+21,\ldots,56a+12\}\cup \{64a+27 ,\ldots,72a+18\},
\end{array}\\
\begin{array}{rcl}
\cup_{j=0}^{a-3}\nn{\cup_{i=3}^6 R_{2i+8j}} & =& \{3,\ldots,2a-2\}\cup\{4a+5,\ldots,6a\} \cup \{12a+8, \ldots,16a-1\} \\
&& \cup \{ 16a+15,\ldots,20a+6\} \cup \{25a+13, \ldots, 26a+10\}\cup \\
&&\{ 27a+14,\ldots,28a+11\} \cup \{ 28a+14,\ldots,29a+11\}\cup \\
&&\{30a+14,\ldots, 31a+11\} \cup \{33a+13,\ldots,34a+10\}\cup \\
&&\{ 35a+15,\ldots,36a+12\} \cup  \{ 36a+14,\ldots,37a+11\} \cup \\
&&\{38a+15,\ldots,39a+12\}\cup \{ 40a+19,\ldots,48a+2\}\cup \\
&&\{ 56a+25,\ldots,64a+8\},
\end{array}\\
\begin{array}{rcl}
\nn{\cup_{j=1}^{4}R_j \cup R_{8a-4} \cup \cup_{t=8a}^{8a+3} R_t} & =&
\{1,2a-1,2a,4a,4a+1,6a+1, 6a+2 \}\cup \{8a+2,\ldots,\\
&& 8a+5 \}\cup \{12a+4, \ldots,12a+7 \} \cup
\{16a+2,\ldots,\\&&16a+8\} \cup \{20a+7,\ldots,
20a+10\} \cup
\{24a+9, 24a+10,\\
&& 24a+11,25a+12,26a+12,26a+13,27a+13, \\
&&29a+12, 30a+12,31a+12, 33a+12,34a+13, \\
&& 35a+14,36a+13, 37a+12, 39a+13\} \cup  \{40a+13,\ldots,\\
&& 40a+18\} \cup \{48a+9, 48a+10\} \cup \{48a+13,  \ldots,\\
&& 48a+20\} \cup \{56a+17,\ldots, 56a+24\} \cup \{64a+15,\\
&& 64a+16\} \cup \{64a+19,\ldots, 64a+26\} \cup  \{72a+23,\ldots,\\
&& 72a+27\},
\end{array}
\end{gather*}

\end{footnotesize}
\noindent  $H$ satisfies conditions (a) and (b) of Definition \ref{def:H}. Now, we list the partial
sums for each row. We have
\begin{footnotesize}
\begin{eqnarray*}
\S(R_1) & = & ( 2+8a,       25+ 72a,      16+  48a, -8-16a,        5+ 24a,        -9-16a,       -26-64a, \\&& -18-48a,     0 ),\\
 \S(R_2) & = & ( 16+40a,       15+ 34a,       42+106a, 35+86a,        13+30a,        25+61a,       13+ 36a,\\&&  -8-20a,
     0 ),\\
\S(R_3) & =&   ( 4+8a,        24+56a,      24+  60a, 49+124a,       44+116a,      18+  52a,        31+78a,\\&&  19+48a, 
0 ),\\
     \end{eqnarray*}
  \end{footnotesize}
  \begin{footnotesize}
  \begin{eqnarray*}
\S(R_4) & =&   (-17-40a,       -11-28a, 7+12a, 6+14a, 29+70a,     22+ 58a,         -2+2a,\\&&  12+37a,0),\\
\S(R_{5+8i}) & =&   (-12-24a-i,   -33-72a-9i,  -25-48a-13i, -3-5i, -4-8a-3i,   \\&& 23+56a+5i,    16+32a+9i, -12-32a+i,0
),\\
\S(R_{6+8i}) & =&   (-13-33a-i,   -32-73a-9i,  -26-53a-13i, -6-13a-5i,    -6-19a-3i,\\&&   19+37a+ 5i,    14+17a+9i, -12-39a+i,        0 ),\\
\S(R_{7+8i}) & =&   (-12-38a+i,  -35-86a -7i,  -29-78a -3i, -5-30a+5i,    -6-26a+ 3i,\\&&   23+38a+11i,  16+30a+
7i,-14-34a -i,0 ),\\
\S(R_{8+8j}) & =&   (-11-29a+j,   -32-69a-7j, -24-57a  -3j, -2-17a+5j,     -4-15a+3j, \\&& 23+41a+ 11j,    14+29a+7j, -14-27a-j,              0 ),\\
\S(R_{9+8j}) & =&   (-13-32a-j,  -38-80a -9j, -32-56a -13j, -6-8a-5j, -6-16a-3j,\\&& 25+48a+5j, 20+24a+ 9j,-12-40a+ j,
           0 ),\\
\S(R_{10+8j}) & =&   (-13-25a-j, -36-65a  -9j, -32-45a -13j, -8-5a-5j,-7-11a -3j, \\&& 22+45a+   5j, 19+25a+   9j,-11-31a+ j,0 ),\\
\S(R_{11+8j}) & =&   (-11-30a+j, -38-78a -7j,-30-70a -3j, -2-22a+5j, -4-18a+ 3j,\\&& 29+46a+ 11j, 20+38a+7j, -14-26a-j, 
  0 ),\\
\S(R_{12+8h})& =&  (  -11-37a+h,   -36-77a-7h,  -26-65a -3h, -25a+5h, -3-23a+3h,  \\&& 28+33a+11h,   17+21a+ 7h,
-15-35a-h,              0 ), \\
\S(R_{8a-4}) & =&   (-13-36a,     -22  -84a,    -20   -68a,-10 -20a,  -11 -24a,       4+ 40a,     1+   24a,\\
&&-15 -40a,     0 ), \\
\S(R_{8a}) & =&   (11+24a,              1,       -12-48a,-8 -32a,    6+16a,         7+16a,       26+ 80a,20+ 64a,               0 ), \\
\S(R_{8a+1}) &= &  ( -24-72a,     -11  -33a,    -23-66a,-40 -122a,     -30 -102a,   -12-46a,  \\&&   -14  -52a, 9+20a,       0 ),\\
\S(R_{8a+2})& = & (-5-16a,       -27-80a,       -14-46a,-26 -72a,-41 -120a,      -34-104a, \\&&    -18  -56a,-21 -64a,
0 ), \\
\S(R_{8a+3})& = & (  25+72a,        20+60a,       -6 -12a,7+ 15a,-5 -14a, -24-70a,-20-58a, \\
&& -2a,  0 ).\\
\end{eqnarray*}

\end{footnotesize}
\noindent By a long direct calculation, the reader can check that the elements of each $\S(R_t)$ are pairwise distinct
both modulo $18(8a+3)+1$ and modulo $18(8a+3)+2$
and in particular each row sums to $0$. From the definition of $H$ we obtain the following expression of the columns:
\begin{footnotesize}
\begin{eqnarray*}
C_1 & = & (2+8a,16+40a,4+8a,-17-40a,-12-24a,\q^{8a-6}, 11+24a,-24-72a, \\&&-5-16a,25+72a)^T,\\
C_2 & = & ( 23+64a,-1-6a,20+48a,6+12a,-21-48a,-13-33a,\q^{8a-6},13+39a,\\&& -22-64a,-5-12a)^T,\\
C_3 & = & (-9-24a,27+72a,4a,18+40a,8+24a,-19-40a,-12-38a, \q^{8a-6}, \\
&&13+34a, -26-72a )^T,\\
C_4 & = & (-24-64a,-7-20a,25+64a,-1+2a,22+48a,6+20a,-23-48a,\\
&&-11-29a,\q^{8a-6}, 13+27a)^T,\\
C_{5+8i} & =& (\q^{8i}, 13+40a-i, -22-56a-8i, -5-8a-4i,23+56a+8i,-1-8a+2i,\\&& 20+40a+8i,
6+8a+4i, -21-40a-8i, -13-32a-i,\q^{8a-6-8i} )^T,\\
C_{6+8i} & = & (\q^{1+8i},12+31a-i,-26-64a-8i,-7-12a-4i,27+64a+8i,-6a+2i ,\\&& 24+48a+8i,8+12a+4i,-25-48a-8i,-13-25a-i,
\q^{8a-7-8i} )^T,\\
     \end{eqnarray*}
  \end{footnotesize}
  \begin{footnotesize}
  \begin{eqnarray*}
C_{7+8i} & = & (\q^{2+8i},13+26a+i,-24-56a-8i,-7-24a+4i,25+56a+8i,-1+4a-2i,\\&&
22+40a+8i,6+24a-4i,-23-40a-8i,-11-30a+i,\q^{8a-8-8i} )^T,\\
C_{8+8j} & = & (\q^{3+8j},14+35a+j,-28-64a-8j,-5-20a+4j,29+64a+8j,-2+2a-2j,\\&&
26+48a+8j,4+20a-4j,-27-48a-8j,-11-37a+j,\q^{8a-9-8j})^T,\\
C_{9+8j} & = & (\q^{4+8j},12+32a-j,-26-56a-8j,-7-8a-4j,27+56a+8j,-8a+2j,\\&&24+40a+8j,8+8a+4j,-25-40a-8j,-13-24a-j,\q^{8a-10-8j})^T,\\
C_{10+8j} & = & (\q^{5+8j},12+39a-j,-30-64a-8j,-9-12a-4j,31+64a+8j, 1-6a+2j, \\&&
28+48a+8j,10+12a+4j,-29-48a-8j,-14-33a-j,\q^{8a-11-8j} )^T,\\
C_{11+8j} & = & (\q^{6+8j}, 14+34a+j, -28-56a-8j, -5-24a+4j, 29+56a+8j, -2+4a-2j,\\&& 26+40a+8j, 4+24a-4j,
  -27-40a-8j, -11-38a+j,\q^{8a-12-8j} )^T,\\
C_{12+8h} & = & (\q^{7+8h}, 14+27a+h, -32-64a-8h, -3-20a+4h, 33+64a+8h, \\
&& -3+2a-2h, 30+48a+8h, 2+20a-4h, -31-48a-8h, -10-29a+h,\\
&&\q^{8a-13-8h})^T,\\
C_{8a-4} & = & (\q^{8a-9},12+28a, -16-72a, -11-16a, 17+72a, -1-4a, 14+56a, 10+16a,\\&&
 -15-56a,  -10-24a, \q^3 )^T,\\
C_{8a} &= & ( -14-40a,\q^{8a-6},15+40a, -20-72a, -9-16a, 21+72a, 1, 18+56a,\\&& 7+16a, -19-56a )^T,\\
C_{8a+1} & =& ( -17-48a,-12-25a,\q^{8a-6},13+31a, -18-64a, -3-12a, 19+64a,\\&& -2-6a, 16+48a, 4+12a )^T,\\
C_{8a+2} & = & (8+16a, -21-56a, -12-30a,\q^{8a-6},13+38a, -22-72a, -6-16a, 23+72a,\\&& -3-8a, 20+56a  )^T,\\
C_{8a+3} & =&  (18+48a, 8+20a, -19-48a, -12-37a,\q^{8a-6},13+35a, -20-64a,\\&& -9-20a, 21+64a, 2a  )^T,
 \end{eqnarray*}

\end{footnotesize}
\noindent where $i=0,\ldots,a-1$, $j=0,\ldots,a-2$ and $h=0,\ldots, a-3$. Each column contains $9$ elements, hence condition (c) holds.
The partial sums of the columns are:
\begin{footnotesize}
\begin{eqnarray*}
\S(C_1) & = & ( 2+8a,       18+ 48a,        22+56a, 5+16a,-7-8a,        4+ 16a,      -20 -56a,\\&& -25 -72a,     0 ),\\
 \S(C_2) & = & (  23+64a,   22+     58a,     42+  106a, 48+118a, 27+70a,      14+  37a,     27+   76a, \\&& 5+12a,
  0 ),\\
  \S(C_3) & =&   (-9-24a,       18+ 48a,       18+ 52a,36+ 92a, 44+116a,        25+76a,      13+  38a, \\&& 26+72a,
0 ),\\
\S(C_4) & =&   (-24-64a,       -31-84a,    -6    -20a,-7 -18a, 15+30a,      21+  50a,       -2+   2a,\\&&  -13-27a, 0),\\
\S(C_{5+8i}) & =&   (13+40a-i, -9-16a   -9i,-14-24a  -13i,9+32a -5i, 8+24a-3i,\\&&   28+64a+  5i,  34+72a+  9i,13+32a+ i,   0 ),\\
\S(C_{6+8i}) & =&   ( 12+31a -i, -14-33a  -9i,-21-45a  -13i, 6+19a-5i,6+13a-3i,\\&& 30+61a+    5i, 38+73a+   9i,
13+25a+i,       0 ),\\
\S(C_{7+8i}) & =&   (13+26a+i,   -11-30a-7i, -18 -54a -3i, 7+2a+5i,6+6a+3i, \\&&  28+46a+ 11i,  34+70a+  7i,11+30a -i, 0 ),\\
\S(C_{8+8j}) & =&   ( 14+35a+j, -14-29a  -7j, -19-49a  -3j, 10+15a+5j,8+17a+3j,\\&&  34+65a+  11j, 38+85a+ 7j,
11+37a-j,             0 ),\\
\S(C_{9+8j}) & =&   (  12+32a-j, -14-24a  -9j, -21-32a -13j,6+24a -5j, 6+16a-3j, \\&&   30+56a+ 5j,  38+64a+  9j,
13+24a+j,            0 ),\\
     \end{eqnarray*}
  \end{footnotesize}
  \begin{footnotesize}
  \begin{eqnarray*}
\S(C_{10+8j}) & =&   (12+39a-j, -18-25a   -9j,-27-37a  -13j,4+27a -5j,5+21a-3j,
\\&&  33+69a+   5j, 43+81a+   9j, 14+33a+j,   0 ),\\
\S(C_{11+8j}) & =&   (  14+34a+j, -14 -22a -7j, -19-46a  -3j,10+10a+ 5j,8+14a+3j, \\&& 34+54a+  11j,  38+78a+  7j, 11+38a-j,    0 ),\\
\S(C_{12+8h})& =&  ( 14+27a+h, -18-37a  -7h,  -21-57a -3h,12+7a+ 5h, 9+9a+3h, \\&& 39+57a+  11h,   41+77a+ 7h, 10+29a-h,                0 ), \\
\S(C_{8a-4}) & =&   (  12+28a,-4  -44a,     -15  -60a,2+ 12a,1+8a,    15+    64a,       25+ 80a,\\&&  10+24a,   0 ), \\
\S(C_{8a}) & =&   (  -14-40a,              1,      -19 -72a, -28-88a,-7-16a,       -6 -16a,        12+40a, \\&& 19+56a, 
        0 ), \\
\S(C_{8a+1}) &= &  (  -17-48a,      -29 -73a,       -16-42a, -34-106a,-37-118a,-18-54a,    \\&& -20  -60a,  -4-12a,
0 ),\\
\S(C_{8a+2})& = & (8+16a,-13-40a,  -25-70a,-12 -32a,-34-104a,   -40-120a,  \\&&   -17  -48a, -20 -56a,    0 ), \\
\S(C_{8a+3})& = & ( 18+48a,      26+  68a,       7+  20a, -5-17a, 8+18a,-12-46a,      -21 -66a, \\&& -2a, 0 ). \\
\end{eqnarray*}

\end{footnotesize}
\noindent Since each column sums to $0$, $H$ is a H$(8a+3;9)$. Finally, one can check that the elements of each $\S(C_{t})$ are pairwise distinct both modulo $18(8a+3)+1$
and modulo $18(8a+3)+2$. We conclude that $H$ is an $\SH^*(8a+3;9)$.	\\
\noindent \underline{Case 2.} Let $n=8a+7$ and  $H$ be the $n\times n$ array whose rows $R_t$ are defined as follows:
\begin{footnotesize}
\begin{eqnarray*}
R_1 & = & (6+8a,55+64a,-21-24a,-56-64a,33+40a,\q^{8a-2},-34-40a,\\
&&-41-48a,16+16a,42+48a),\\
R_2 & = & (36+40a,-4-6a,63+72a,-17-20a,-50-56a,32+39a,\q^{8a-2},\\
&&-29-33a, -49-56a,18+20a ),\\
R_3 & = & (8+8a,44+48a,2+4a,57+64a,-9-8a,-58-64a,26+26a,\q^{8a-2}, \\
&& -27-30a,-43-48a ),\\
R_4 & = & (-37-40a,12+12a,38+40a,2a,51+56a,-13-12a,-52-56a,\\
&& 27+27a,\q^{8a-2},-26-29a),\\
R_{5+8i} & =& (\q^{8i}, -24-24a-i,-45-48a-8i,20+24a-4i, 46+48a+8i,-5-8a+2i,\\
&&59+64a+8i,-19-24a+4i,-60-64a-8i,28+32a-i,\q^{8a-2-8i} ),\\
 R_{6+8i} & = & (\q^{1+8i},-25-25a-i,-39-40a-8i,16+20a-4i,40+40a+8i,-3-6a+2i,\\
&&53+56a+8i, -15-20a+4i,-54-56a-8i,27+31a-i,  \q^{8a-3-8i} ),\\
R_{7+8i} & = & (\q^{2+8i},-31-38a+i,-47-48a-8i,10+8a+4i, 48+48a+8i, 1+4a-2i,\\
&&61+64a+8i,-11-8a-4i,-62-64a-8i,31+34a+i,\q^{8a-4-8i} ),\\
R_{8+8j} & = & (\q^{3+8j},-30-37a+j, -41-40a-8j,14+12a+4j,42+40a+8j,-1+2a-2j,\\
&&55+56a+8j,-15-12a-4j,-56-56a-8j,32+35a+j,\q^{8a-5-8j}),\\
R_{9+8i} & =&(\q^{4+8i},-29-32a-i,-49-48a-8i,18+24a-4i,50+48a+8i,-4-8a+2i,\\
&&63+64a+8i,-17-24a+4i,-64-64a-8i,32+40a-i,\q^{8a-6-8i}),\\
R_{10+8i} & = & (\q^{5+8i}, -30-33a-i,-43-40a-8i,14+20a-4i,44+40a+8i,-2-6a+2i,\\
&&57+56a+8i,-13-20a+4i,-58-56a-8i,31+39a-i,\q^{8a-7-8i} ),\\
R_{11+8i} & = & (\q^{6+8i},-26-30a+i,-51-48a-8i,12+8a+4i, 52+48a+8i,4a-2i,\\
&&65+64a+8i,-13-8a-4i,-66-64a-8i,27+26a+i, \q^{8a-8-8i} ),\\
     \end{eqnarray*}
  \end{footnotesize}
  \begin{footnotesize}
  \begin{eqnarray*}
R_{12+8j} & = & (\q^{7+8j},-25-29a+j,-45- 40a-8j,16+12a+4j, 46+40a+8j,-2+2a-2j,\\
&&59+56a+8j,-17-12a-4j, -60-56a-8j,28+27a+j, \q^{8a-9-8j}),\\
R_{8a} & = & (\q^{8a-5},-31-36a,-33-48a,10+16a,34+48a, -3-4a,47+64a,-11-16a,\\
&&-48-64a,35+40a,\q^3),\\
R_{8a+4} &= & (23+24a,\q^{8a-2},-22-24a,-37-48a,12+16a,38+48a,1,51+64a,\\
&&-14-16a,-52-64a ),\\
R_{8a+5} & =& (-60-72a,28+31a,\q^{8a-2},-24-25a,-45-56a,20+20a,46+56a,-5-6a,\\
&&59+72a,-19-20a),\\
R_{8a+6} & = & (-13-16a,-54-64a,30+34a,\q^{8a-2},-25-26a,-39-48a,15+16a,\\
&&40+48a, -7-8a,53+64a ),\\
R_{8a+7} & =& (61+72a,-11-12a,-62-72a,31+35a,\q^{8a-2},-31-37a,-47-56a,10+12a,\\
&&48+56a,1+2a  ),
 \end{eqnarray*}

\end{footnotesize}
\noindent where $i=0,\ldots,a-1$ and $j=0,\ldots,a-2$. Each row contains $9$ elements.
 Since
\begin{footnotesize}
\begin{gather*}
\begin{array}{rcl}
\nn{\cup_{j=1}^{4}R_j  \cup \cup_{t=8a+4}^{8a+7} R_t} & =&
\{1,2a,2a+1,4a+2,6a+4, 6a+5 \}\cup \{8a+6,\ldots,8a+9 \}\\
&& \cup \{12a+10,\ldots,12a+13 \} \cup \{16a+12,\ldots,16a+16\} \cup \\
&&\{20a+17,\ldots,20a+20\} \cup \{24a+21,24a+22,24a+23,\\
&& 25a+24,26a+25,26a+26,27a+27, 29a+26,30a+27,\\
&& 31a+28,33a+29,34a+30,35a+31,37a+31,39a+32,\\
&&40a+33,40a+34,40a+36,40a+37,40a+38\} \cup \\&&\{48a+37,\ldots, 48a+44\} \cup
   \{56a+45,\ldots,56a+52\} \cup\\&&
  \{64a+51,\ldots, 64a+58\}  \cup  \{72a+59,\ldots,72a+63\},
\end{array}\\
\begin{array}{rcl}
\cup_{j=0}^{a-1}\nn{\cup_{i=2}^5 R_{2i+1+8j}} & =& \{2a+2,\ldots,4a+1\} \cup \{6a+6,\ldots,8a+5\}
\cup \\&&
\{8a+10,\ldots,12a+9\} \cup \{20a+21, \ldots,24a+20 \} \cup \\&&
\{24a+24, \ldots, 25a+23\}\cup \{26a+27,\ldots, 27a+26\}\cup\\&&
\{29a+27,\ldots,30a+26\} \cup \{31a+29 ,\ldots, 33a+28\}\cup \\&&
\{ 34a+31,\ldots,35a+30\}\cup \{ 37a+32,\ldots,38a+31\}\cup \\&&
\{ 39a+33,\ldots,40a+32\}\cup \{ 48a+45,\ldots,56a+44\}\cup \\&&
\{64a+59 ,\ldots,72a+58\},
\end{array}\\
\begin{array}{rcl}
\cup_{j=0}^{a-2}\nn{\cup_{i=3}^6 R_{2i+8j}} & =& \{2,\ldots,2a-1\}\cup\{4a+6,\ldots,6a+3\} \cup \\&&
\{12a+14, \ldots,16a+9\} \cup \{ 16a+21,\ldots,20a+16\} \cup \\&&
\{25a+25, \ldots, 26a+23\}\cup \{ 27a+28,\ldots,29a+25\} \cup \\&&
 \{30a+29,\ldots, 31a+27\} \cup \{33a+30,\ldots,34a+28\}\cup \\&&
 \{ 35a+32,\ldots,36a+30\} \cup \{ 36a+32,\ldots,37a+30\} \cup \\&&
  \{38a+33,\ldots,39a+31\}\cup \{ 40a+39,\ldots,48a+30\}\cup \\&&
  \{ 56a+53,\ldots,64a+44\},
\end{array}\\
\begin{array}{rcl}
\nn{\cup_{t=-1}^{1} R_{8a+2t}} & =&
\{4a+3,4a+4,4a+5,16a+10,16a+11\} \cup \{16a+17 \ldots,16a+20\} \cup \\&&
\{26a+24,30a+28,34a+29,36a+31,38a+32,40a+35\}  \cup \\&&
\{48a+31,\ldots, 48a+36\} \cup
  \{64a+45,\ldots, 64a+50\},
\end{array}
\end{gather*}

\end{footnotesize}
\noindent  $H$ satisfies conditions (a) and (b) of Definition \ref{def:H}. Now, we list the partial
sums for each row. We have

\begin{footnotesize}
\begin{eqnarray*}
\S(R_1) & = & (6+8a, 61+72a,40+48a,-16-16a,17+24a,-17-16a, -58-64a,\\
&&-42-48a,0 ),\\
\S(R_2) & = &(36+40a, 32+34a,95+106a,78+86a,28+30a,60+69a, 31+36a,\\
&&-18-20a,0 ),\\
\S(R_3) & = &(8+8a,52+56a,54+60a,111+124a,102+116a,44+52a,70+78a,\\
&& 43+48a,0 ),\\
   \end{eqnarray*}
 \end{footnotesize}
 \begin{footnotesize}
 \begin{eqnarray*}
\S(R_4) & = &(-37-40a,-25 -28a,13+ 12a,13+14a,64+70a,51+58a,  -1+2a,\\
&&26+29a,0 ),\\
\S(R_{5+8i})&=&  (-24-24a-i,-69-72a-9i,-49-48a-13i,-3-5i, \\
 &&-8-8a-3i,51+ 56a+5i, 32+32a+9i,-28-32a+i,  0 ),\\
\S(R_{6+8i})&=& (-25-25a-i,-64   -65a-9i,-48  -45a-13i,-8     -5a-5i,\\
 &&-11-11a-3i,42+ 45a+5i,  27+  25a+9i,-27     -31a+i,               0 ),\\
\S(R_{7+8i})&=&(-31-38a+i,-78   -86a-7i,-68   -78a-3i,-20   -30a+5i, \\
 &&-19-26a+3i,42+ 38a+11i,31 +   30a+7i,  -31   -34a-i,               0 ),\\
\S(R_{8+8j}) & = & (    -30 -37a+j, -71  -77a-7j,-57   -65a-3j,-15   -25a+5j, \\
 &&-16-23a+3j, 39+   33a+11j,  24+  21a+7j,-32-35a-j,               0 ),\\
\S(R_{9+8i}) & = &(-29-32a-i,-78   -80a-9i,-60  -56a-13i,-10    -8a-5i, \\
 &&-14-16a-3i,49+    48a+5i, 32+   24a+9i, -32    -40a+i,               0 ),\\
\S(R_{10+8i}) & = & (-30-33a-i,-73   -73a-9i,-59  -53a-13i,-15   -13a-5i,\\
 &&-17-19a-3i,40+   37a+5i,27+    17a+9i,-31     -39a+i,               0 ),\\
\S(R_{11+8i}) & = & (-26-30a+i,-77   -78a-7i,-65   -70a-3i,-13   -22a+5i,  \\
 &&-13-18a+3i,52+   46a+11i,39+    38a+7i,-27     -26a-i,               0 ),\\
\S(R_{12+8j}) & = & (-25-29a+j,-70   -69a-7j,-54   -57a-3j,-8    -17a+5j,\\
 &&-10-15a+3j,49+   41a+11j,32+    29a+7j,-28     -27a-j,               0 ),\\
\S(R_{8a}) & = &(-31-36a,-64       -84a,   -54    -68a,-20       -20a,  -23-24a,24+        40a, \\
&&13+       24a,-35       -40a,               0 ),\\
\S(R_{8a+4}) & = &(23+24a,               1,       -36-48a,-24       -32a,14+        16a,       15+16a,66+        80a,\\
&&52+        64a,               0 ),\\
\S(R_{8a+5}) & = &(-60-72a,-32       -41a,-56       -66a,-101     -122a,-81      -102a,    \\
 &&-35-46a,-40       -52a,   19+     20a,               0 ),\\
 \S(R_{8a+6} )& = &(-13-16a,   -67    -80a,-37       -46a,-62       -72a,-101     -120a,  \\
 &&-86-104a,-46       -56a,-53       -64a,               0 ),\\
\S(R_{8a+7} )& = &(61+72a,50+        60a,-12       -12a,     19+   23a,-12       -14a, -59-70a, \\
 &&-49       -58a,-1         -2a,               0).
  \end{eqnarray*}

\end{footnotesize}
\noindent By a long direct calculation, the reader can check that the elements of each $\S(R_t)$ are pairwise distinct
both modulo $18(8a+7)+1$ and modulo $18(8a+7)+2$
and in particular each row sums to $0$. From the definition of $H$ we obtain the following expression of the columns:
\begin{footnotesize}
\begin{eqnarray*}
C_1 & = & ( 6+8a,36+ 40a,8+8a,-37 -40a, -24-24a, \q^{8a-2}, 23+24a, -60-72a,\\
&&-13-16a,  61+72a )^T,\\
C_2 & = & (55+64a,-4-6a,44+48a,12+12a,-45-48a,-25-25a,\q^{8a-2},28+31a,\\
&&-54-64a,-11-12a )^T,\\
C_3 & = & (-21-24a,63+ 72a,2+4a,38+40a,20+ 24a,-39-40a,-31-38a,\q^{8a-2},\\
&& 30+34a,-62-72a  )^T,\\
C_4 & = & (-56-64a, -17-20a,57+ 64a, 2a,46+ 48a,16+20a, -47-48a, -30-37a,\\
&& \q^{8a-2},31+35a )^T,\\
C_{5+8i} & = & (\q^{8i}, 33+40a-i,-50-56a-8i,-9-8a-4i,51+56a+8i,-5-8a+2i,\\
&& 40+40a+8i,10+8a+4i,-41-40a-8i,-29-32a-i,\q^{8a-2-8i})^T,\\
C_{6+8i} & = & (\q^{8i+1},32+39a-i,-58-64a-8i,-13-12a-4i,59+64a+8i,-3-6a+2i,\\
&&48+48a+8i,14+12a+4i,-49-48a-8i,-30-33a-i,\q^{8a-3-8i})^T,\\
 \end{eqnarray*}
 \end{footnotesize}
 \begin{footnotesize}
 \begin{eqnarray*}
C_{7+8i} & = &(\q^{8i+2},26+26a+i,-52-56a-8i,-19-24a+4i,53+56a+8i,1+4a-2i,\\
&&42+40a+8i,18+24a-4i,-43-40a-8i,-26-30a+i,\q^ {8a-4-8i} )^T,\\
C_{8+8j} & = & (\q^{8j+3},27+27a+j,-60-64a-8j,-15-20a+4j,61+64a+8j,-1+2a-2j,\\
&&50+48a+8j,14+20a-4j,-51-48a-8j, -25-29a+j,\q^{8a-5-8j} )^T,\\
C_{9+8i} & = & (\q^{8i+4},28+32a-i,-54-56a-8i, -11-8a-4i,55+56a+8i,-4-8a+2i,\\
&& 44+40a+8i,12+8a+4i,-45-40a-8i,-25-24a-i,\q^{8a-6-8i} )^T,\\
C_{10+8i} & = & ( \q^{8i+5},27+31a-i,-62-64a-8i,-15-12a-4i,63+64a+8i,-2-6a+2i,\\
&& 52+48a+8i,16+12a+4i,-53-48a-8i,-26-25a-i,\q^{8a-7-8i})^T,\\
C_{11+8i} & = & (\q^{8i+6}, 31+34a+i,-56-56a-8i,-17-24a+4i,57+56a+8i, 4a-2i,\\
&&46+40a+8i,16+24a-4i,-47-40a-8i,-30-38a+i,\q^{8a-8-8i} )^T,\\
C_{12+8j} & = & (\q^{8j+7},32+35a+j, -64-64a-8j, -13-20a+4j, 65+64a+8j,-2+2a-2j,\\
 &&54+48a+8j,12+ 20a-4j, -55-48a-8j, -29-37a+j,		\q^{8a-9-8j} )^T,\\
C_{8a} & = & (\q^{8a-5}, 26+28a, -52-72a, -19-16a,53+72a, -3-4a,42+56a,18+16a,\\
&&-43-56a,-22-24a,\q^{3} )^T,\\
C_{8a+4} & = & (-34-40a,\q^{8a-2},35+40a,-56-72a,-17-16a, 57+72a, 1, 46+56a, \\
&&15+ 16a, -47-56a )^T,\\
C_{8a+5} & = & (-41-48a,-29-33a, \q^{8a-2}, 33+39a,-50-64a,-9-12a, 51+64a,-5-6a,\\
&&40+ 48a,10+ 12a  )^T,\\
C_{8a+6} & = & ( 16+16a,-49-56a,-27-30a,\q^{8a-2},32+38a, -58-72a, -14-16a,59+72a, \\
&&-7-8a, 48+56a )^T,\\
C_{8a+7} & = & (42+48a,18+20a, -43-48a, -26-29a,\q^{8a-2},26+ 27a,-52-64a, -19-20a,\\
&& 53+64a,1+ 2a )^T,
\end{eqnarray*}

\end{footnotesize}
\noindent where $i=0,\ldots,a-1$ and $j=0,\ldots,a-2$. Clearly, each column has exactly $9$ filled cells.
We have

\begin{footnotesize}
\begin{eqnarray*}
\S(C_1) & = & (6+8a,       42+ 48a, 50+       56a,13+        16a,-11        -8a,12+        16a,
-48-56a,\\&&-61       -72a,               0 ),\\
\S(C_2) & = & (55+64a,51+        58a, 95+      106a,107+      118a, 62+       70a, 37       +45a,
65+76a,\\&&11+        12a,               0  ),\\
\S(C_3) & = & (-21-24a,42+        48a,44+        52a,82   +     92a,102+      116a,63+        76a,
32+38a,\\&&62+        72a,               0 ),\\
\S(C_4) & = & ( -56-64a,-73       -84a,-16       -20a,-16       -18a,30+        30a,46+        50a,
-1+2a,\\&&-31       -35a,               0 ),\\
 \S(C_{5+8i}) & = & (33+40a -i,-17-16a   -9i,-26-24a  -13i,25+32a   -5i,  20+24a -3i,\\
&&60+64a+5i,70+72a+9i,29+32a+      i,               0 ),\\
\S(C_{6+8i}) & = & (32+39a-i,-26-25a   -9i,-39-37a  -13i,20+27a   -5i,  17+21a -3i,\\
&&65+69a   + 5i, 79+81a+9i,30+33a+      i,               0 ),\\
\S(C_{7+8i}) & = & ( 26+26a+i,-26-30a   -7i,-45-54a   -3i,8+2a+      5i,    9+6a+  3i, \\
&&51+46a+  11i,69+70a+7i,26+30a     -i,               0  ),\\
\S(C_{8+8j}) & = & (27+27a+j,-33-37a  -7j,-48-57a   -3j, 13+7a+    5j,12+9a+     3j, \\
&& 62+57a+11j, 76+77a+ 7j,25+29a     -j,               0 ),\\
\S(C_{9+8i}) & = & (28+32a-i,   -26-24a-9i,-37-32a  -13i,18+24a   -5i,  14+16a -3i,\\
&&58+56a+    5i,70+64a+9i,25+24a+      i,               0  ),\\
   \end{eqnarray*}
 \end{footnotesize}
 \begin{footnotesize}
 \begin{eqnarray*}
\S(C_{10+8i}) & = & (27+31a-i,-35-33a   -9i,-50-45a  -13i,13+19a   -5i,  11+13a -3i, \\
&& 63+61a+    5i,79+73a+9i,26+25a+      i,               0 ),\\
\S(C_{11+8i}) & = & (31+34a+i,-25-22a   -7i,-42-46a   -3i,15+10a+    5i,   15+14a+ 3i, \\
&&61+54a+  11i,77+78a+ 7i,30+38a     -i,               0  ),\\
\S(C_{12+8j}) & = & (32+35a+ j,-32-29a   -7j,-45-49a   -3j,20+15a+    5j,   18+17a+ 3j,\\
&&72+65a+   11j,84+85a+7j,29+37a     -j,               0  ),\\
\S(C_{8a}) & = & ( 26+28a,-26       -44a,-45       -60a, 8+        12a,5+          8a,47+        64a,
65+80a, \\&&       22+24a,               0 ),\\
\S(C_{8a+4}) & = & ( -34-40a,               1,    -55   -72a,-72       -88a,-15       -16a,-14       -16a,
32+40a, \\&&     47+  56a,               0 ),\\
\S(C_{8a+5}) & = & (-41-48a,   -70    -81a,-37       -42a,-87      -106a,      -96-118a,-45 -54a,\\
&&-50-60a,-10       -12a,               0  ),\\
\S(C_{8a+6}) & = & (16+16a,-33       -40a,-60       -70a,-28       -32a,     -86 -104a, -100     -120a, \\
&&-41-48a,-48       -56a,               0  ),\\
\S(C_{8a+7}) & = & (42+48a,60+        68a,17+        20a,-9         -9a,17+        18a,-35       -46a,
-54-66a,\\&& -1        -2a,               0 ).
\end{eqnarray*}

\end{footnotesize}
\noindent Since each column sums to $0$, $H$ is an $\H(8a+7;9)$. Also in this case, the elements of each $\S(C_t)$ are pairwise distinct both modulo $18(8a+7)+1$
and modulo $18(8a+7)+2$. We conclude that $H$ is an $\SH^*(8a+7;9)$.
\end{proof}

\begin{prop}\label{prop:10}
Let $n\geq 10$ be even. Then, there exists an $\SH^*(n;10)$.
\end{prop}

\begin{proof}
An $\SH^*(10;10)$  can be found in \cite{web}. So let $n\geq 12$ be even and let
$H$ be the $n\times n$ partially filled array whose rows $R_t$ are the following:
\begin{footnotesize}
\begin{eqnarray*}
R_1 & = & ( 16-10n, 10n-9,12-10n,6-10n,8-10n,1-10n,10n-4, \q, \\
&&10n,\q^{n-12},10n-17,\q,10n-13 ), \\
R_{2+2i} & = & (\q^{2i+1}, -21-20i,\q,-25-20i,22+20i,-29-20i,26+20i,30+20i, \\
{} & {} & 33+20i,37+20i,-35-20i,\q,-38-20i,\q^{n-13-2i}),\\
R_3 & = & (\q,7,-4,11,-8,-14,-12,-19,16,\q,20, \q^{n-12},3),\\
R_{5+2i} & =  & (\q^{2i+1}, 23+20i,\q,27+20i,-24-20i,31+20i,-28-20i,-34-20i, \\
{} & {} & -32-20i,-39-20i,36+20i,\q,40+20i,\q^{n-13-2i}),\\
R_{n-10} & = &  (82-10n,\q^{n-12},99-10n,\q,95-10n,-98+10n,91-10n, -94+10n,\\
&&-90+10n,-87+10n,-83+10n,85-10n,\q),\\
R_{n-8} & = &  (65-10n,\q,62-10n,\q^{n-12},79-10n,\q,75-10n,-78+10n, 71-10n,\\
&&-74+10n,-70+10n,-67+10n,-63+10n),\\
R_{n-7} & = &  (-80+10n,\q^{n-12},-97+10n,\q,-93+10n,96-10n,-89+10n,\\
{} & {} & 92-10n,86-10n,88-10n,81-10n,-84+10n,\q),\\
R_{n-6} & = &  (-47+10n,-43+10n,45-10n,\q,42-10n,\q^{n-12},59-10n,\q,\\
{} & {} & 55-10n,-58+10n,51-10n,-54+10n,-50+10n),\\
R_{n-5} & = &  (-64+10n,\q,-60+10n,\q^{n-12},-77+10n,\q,-73+10n,76-10n,\\
{} & {} & -69+10n,72-10n,66-10n,68-10n,61-10n),\\
R_{n-4} & = &  (-34+10n,-30+10n,-27+10n,-23+10n,25-10n,\q,22-10n,\\
{} & {} & \q^{n-12},39-10n,\q,35-10n,-38+10n,31-10n),\\
R_{n-3} & = &  (48-10n,41-10n,-44+10n,\q,-40+10n,\q^{n-12},-57+10n,\q,\\
{} & {} & -53+10n,56-10n,-49+10n,52-10n,46-10n),\\
              \end{eqnarray*}
		 \end{footnotesize}
 \begin{footnotesize}
 \begin{eqnarray*}
R_{n-2} & = & (-18+10n,11-10n,-14+10n,-10+10n,-7+10n,-3+10n,\\
{} & {} & 5-10n,\q,2-10n,\q^{n-12},19-10n,\q,15-10n),\\
R_{n-1} &= & (32-10n, 26-10n, 28-10n, 21-10n,-24+10n,\q,-20+10n,  \q^{n-12},\\
&&-37+10n,\q, -33+10n,36-10n,-29+10n),\\
R_{n} & = & (\q,-5,2,-9,6,10,13,17,-15,\q,-18,\q^{n-12},-1),
\end{eqnarray*}
\end{footnotesize}

\noindent where $i=0,\ldots,\frac{n-14}{2}$. Note that every row contains exactly $10$ filled cells. Also,
it is easy to see that $\nn{R_{3}\cup R_{n}}=\{1,\ldots,20\}$,
$\nn{R_{1}\cup R_{n-2}} =\{10n-19,\ldots,10n\}$ and
$\nn{R_{2h}\cup R_{2h+3}}=\{1+20h,\ldots,20+20h\}$
for all $h=1,\ldots,\frac{n-4}{2}$.
Hence, $H$ satisfies conditions (a) and (b) of Definition \ref{def:H}.
Now, we list the partial sums of each row. We have
\begin{footnotesize}
\begin{eqnarray*}
\S(R_1) & = &(16-10n,7,19-10n,25-20n,33-30n,34-40n,30-30n,  30-20n, \\
&&13+10n,0),\\
\S(R_{2+2i}) & =  & (-21-20i,-46-40i,-24-20i,-53-40i,-27-20i,3,36+20i, \\
&& 73+40i,38+20i, 0),\\
\S(R_3) & = & (7,3,14,6,-8,-20,-39,-23,-3,0),\\
\S(R_{5+2i}) & = & (23+20i, 50+40i,26+20i,57+40i,29+20i,-5,-37-20i,-76-40i,\\
&&-40-20i, 0) ,\\
 \S(R_{n-10}) & = & (82-10n,181-20n,276-30n,178-20n,269-30n,175-20n,\\
&& 85-10n,-2,-85+10n, 0),\\
\S(R_{n-8}) & = & (65-10n,127-20n,206-30n,281-40n,203-30n,274-40n,\\
&& 200-30n,130-20n,63-10n,0 ), \\
 \S(R_{n-7}) &= & (-80+10n,-177+20n,-270+30n,-174+20n,-263+30n,\\
{} & {} &  -171+20n,-85+10n,3,84-10n,0),\\
\S(R_{n-6}) & = & (-47+10n,-90+20n,-45+10n,-3,56-10n,111-20n, 53-10n,\\
&&104-20n,50-10n,0),\\
\S(R_{n-5}) & = & (-64+10n,-124+20n,-201+30n,-274+40n,-198+30n,\\
{} & {} &  -267+40n,-195+30n,-129+20n,-61+10n,0),\\
\S(R_{n-4}) & = & (-34+10n,-64+20n,-91+30n,-114+40n,-89+30n, -67+20n,\\
&&-28+10n,7,-31+10n,0),\\
\S(R_{n-3}) & = & (48-10n,89-20n,45-10n,5,-52+10n,-105+20n,-49+10n,\\
&& -98+20n,-46+10n,0),\\
\S(R_{n-2}) & = & (-18+10n,-7,-21+10n,-31+20n,-38+30n,-41+40n, -36+30n,\\
&&-34+20n,-15+10n,0),\\
\S(R_{n-1}) & = & (32-10n,58-20n,86-30n,107-40n,83-30n,63-20n,26-10n,\\
&&-7,29-10n,0),\\
\S(R_{n}) & = & (-5,-3,-12,-6,4,17,34,19,1,0).
\end{eqnarray*}

\end{footnotesize}
\noindent Clearly, each row sums to $0$.
By a long direct check one can see that the elements of each $\S(R_t)$ are pairwise distinct modulo $20n+1$ for any even
$n\geq 12$. Also,  they are pairwise distinct modulo $20n+2$ if $n\equiv 0 \pmod6$, while, if $n\equiv 2\pmod 6$ and
$n\geq 32$, the partial sums $26+20i$ and $-76-40i$ of $\S(R_{5+2i})$ are equivalent modulo $20n+2$ when
$i=\frac{n-5}{3}$
and
if $n\equiv 4\pmod 6$ and $n\geq 34$ the partial sums $-46-40i$ and $36+20i$ of $\S(R_{2+2i})$ are equivalent modulo
$20n+2$ when $i=\frac{n-4}{3}$.
Moreover,  the following  partial sums $s$ and $s'$ of $\S(R_t)$ are congruent modulo $20n+2$:

$$
\begin{array}{cccc|cccc}
n & t & s & s' & n & t & s & s' \\ \hline
14 & n-3 &  89-20n & -49+10n &
22 & n-8 & 281-40n & 63-10n\\
16 & n-6  & -47+10n & 111-20n &
26 & n-7 & -174+20n & 84-10n \\
20 & n-5 & -198+30n & 0 &
28 & n-10 & 276-30n & -2
\end{array}
$$

From the definition of $H$ we obtain the following expression of its columns:
\begin{footnotesize}
\begin{eqnarray*}
C_1 & = & (16-10n,\q^{n-12},82-10n,\q,65-10n,-80+10n,-47+10n,\\
{} & {} &  -64+10n,-34+10n,48-10n,-18+10n,32-10n,\q)^T,\\
C_2 & = & (-9+10n,-21,7,\q,23,\q^{n-12},-43+10n,\q,-30+10n,41-10n,\\
&& 11-10n,26-10n,-5)^T, \\
C_3 & = & (12-10n,\q,-4,\q^{n-12},62-10n,\q,45-10n,-60+10n,-27+10n,\\
{} & {} &  -44+10n,-14+10n,28-10n,2)^T,\\
C_4 & = & (6-10n,-25,11,-41,27,\q,43,\q^{n-12},-23+10n,\q,-10+10n, 21-10n,-9)^T,\\
C_5 & = & (8-10n,22,-8,\q,-24,\q^{n-12},42-10n,\q,25-10n,-40+10n,-7+10n,\\
&&-24+10n,6)^T,\\
C_6 & = & (1-10n,-29,-14,-45,31,-61,47,\q,63,\q^{n-12},-3+10n,\q,10)^T,\\
C_7 & = & (-4+10n,26,-12,42,-28,\q,-44,\q^{n-12},22-10n,\q,5-10n, -20+10n,13)^T,\\
C_8 & = & (\q,30,-19,-49,-34,-65,51,-81,67,\q,83,\q^{n-12},17)^T,\\
 C_9 & = & (10n,33,16,46,-32,62,-48,\q,-64,\q^{n-12},2-10n,\q,-15)^T,\\
  C_{10+2i} & = & (\q^{2i+1},37+20i,\q,50+20i,-39-20i,-69-20i,-54-20i,  -85-20i,\\
&&71+20i,-101-20i,87+20i,\q,103+20i,\q^{n-13-2i})^T,\\
C_{11} & = & (\q,-35,20,53,36,66,-52,82,-68,\q,-84,\q^{n-12},-18)^T,\\
 C_{13+2i} & = & (\q^{2i+1},-38-20i,\q,-55-20i,40+20i,73+20i,56+20i,86+20i,\\
{} & {} &  -72-20i,102+20i,-88-20i,\q,-104-20i,\q^{n-13-2i})^T,\\
C_{n-2} & = & (-17+10n, \q^{n-12}, -83+10n,\q,-70+10n,81-10n,51-10n,  66-10n,\\
&&35-10n,-49+10n,19-10n,-33+10n,\q)^T,\\
C_{n} & = & (-13+10n,\q,3,\q^{n-12},-63+10n,\q,-50+10n,61-10n,31-10n,\\
{} & {} &  46-10n,15-10n,-29+10n,-1)^T,
\end{eqnarray*}

\end{footnotesize}
\noindent where $i=0,\ldots,\frac{n-14}{2}$.
We observe that each column contains exactly $10$ filled cells, then condition (c) is satisfied.
One can check that the partial sums of the columns are the following:
\begin{footnotesize}
\begin{eqnarray*}
\S(C_1) & = & (16-10n,98-20n,163-30n,83-20n,36-10n,-28,-62+10n, -14,\\
&&-32+10n,0),\\
\S(C_2) & = & (-9+10n,-30+10n,-23+10n,10n,-43+20n,-73+30n,-32+20n,\\
&& -21+10n,5,0), \\
\S(C_3) & = & (12-10n,8-10n,70-20n,115-30n,55-20n,28-10n,-16, -30+10n,\\
&&-2, 0),\\
\S(C_4) & = & (6-10n,-19-10n,-8-10n,-49-10n,-22-10n,21-10n, -2,\\
&&-12+10n,9, 0),\\
\S(C_5) & = & (8-10n,30-10n,22-10n,-2-10n,40-20n,65-30n,25-20n,\\
&&18-10n,-6, 0),\\
\S(C_6) & = & (1-10n,-28-10n,-42-10n,-87-10n,-56-10n,-117-10n,\\
&&-70-10n,-7-10n,-10, 0),\\
 \end{eqnarray*}
 \end{footnotesize}
 \begin{footnotesize}
 \begin{eqnarray*}
\S(C_7) & = & (-4+10n,22+10n,10+10n,52+10n,24+10n,-20+10n,2, 7-10n,\\
&&-13, 0),\\
\S(C_8) & = & (30,11,-38,-72,-137,-86,-167,-100,-17, 0),\\
\S(C_9) & = & (10n,33+10n,49+10n,95+10n,63+10n,125+10n,77+10n,\\
&& 13+10n,15, 0),\\
\S(C_{10+2i}) & = & (37+20i,87+40i,48+20i,-21,-75-20i,-160-40i,-89-20i,\\
&& -190-40i,-103-20i, 0),\\
\S(C_{11}) & = & (-35,-15,38,74,140,88,170,102,18, 0),\\
\S(C_{13+2i}) & = & (-38-20i,-93-40i,-53-20i,20,76+20i,162+40i,90+20i,\\
&& 192+40i,104+20i, 0),\\
\S(C_{n-2}) & = & (-17+10n,-100+20n,-170+30n,-89+20n,-38+10n,28,\\
&& 63-10n,14,33-10n, 0),\\
\S(C_{n}) & = & (-13+10n,-10+10n,-73+20n,-123+30n,-62+20n, -31+10n,\\
&&15,30-10n,1, 0).
\end{eqnarray*}

\end{footnotesize}
\noindent Note that each column sums to $0$, so condition (d) is satisfied,  hence $H$ is an $\H(n;10)$.
Also, again by a direct check, one can see that the elements of each $\S(C_t)$ are pairwise distinct modulo $20n+1$ for any even $n\geq 12$
and that they are pairwise distinct modulo $20n+2$ for all $n\equiv 0,4\pmod 6$.
While, if $n\equiv 2\pmod 6$ and $n\geq 26$, the partial sums $87+40i$ and $-75-20i$ of $\S(C_{10+2i})$ are equivalent
modulo $20n+2$ when $i=\frac{n-8}{3}$.
Also, if $n=14$ the partial sums $-123+30n$ and $15$ of $\S(C_{n})$ are congruent modulo $20n+2$.
Finally, if $n=20$ the partial sums $-170+30n$ and $28$ of $\S(C_{n-2})$ are congruent modulo $20n+2$.
We conclude that $H$ is an $\SH(n;10)$ for any even $n\geq 12$ and that it satisfies also condition $(\ast)$ when
$n\equiv 0\pmod 6$.

Suppose $n\equiv 4\pmod 6$. If $n=16,22,28$, an $\SH^*(n;10)$ can be found in \cite{web}. So let $n=6a+4\geq 34$ and
note that, since $n\equiv 4\pmod 6$, $H$ is an $\SH(n;10)$ whose columns are simple also modulo $20n+2$.
Recall that the only row which is not simple modulo $20n+2$ is $R_{4a+2}$. Interchanging the columns $C_{4a+9}$ (i.e., $C_{13+2i}$ with $i=2a-2$)
and $C_{4a+11}$  (i.e., $C_{13+2i}$ with $i=2a-1$), that are the following:
\begin{footnotesize}
\begin{eqnarray*}
C_{4a+9} & = & (\q^{4a-3},2-40a,\q,-15-40a, 40a, 33+40a, 16+40a, 46+40a, -32-40a,\\
&& 62+40a, -48-40a, \q,-64-40a,\q,\q,\q^{2a-7})^T,\\
C_{4a+11} & = & (\q^{4a-3},\q,\q,-18-40a,\q,-35-40a, 20+40a, 53+40a, 36+40a,   66+40a,\\
&& -52-40a, 82+40a, -68-40a,\q,-84-40a,\q^{2a-7})^T,
\end{eqnarray*}

\end{footnotesize}
\noindent we obtain an $\SH(n;10)$, say $H'$, satisfying also condition $(\ast)$.
In fact, $H'$ is again an $\H(n;10)$ whose columns are simple modulo $20n+1$ and $20n+2$ and so we are left to check its rows.
Looking at the position of the empty cells of the previous two columns, it is easy to see that interchanging them, only the rows
$R_{(4a-3)+1}$, $R_{(4a-3)+3}$, $R_{(4a-3)+4},\ldots, R_{(4a-3)+12}$, $R_{(4a-3)+14}$ change.
Hence, $H'$ is the required array if  these $12$ rows are simple
modulo $20n+1$
and $20n+2$. If $n=34,40,46$ we have checked by computer that $H'$ works (the reader can find these arrays in \cite{web}). For $n \geq 52$, we write explicitly the new rows, that will be denoted by $R'_t$, whose interchanged elements are in bold:
\begin{footnotesize}
  \begin{eqnarray*}
R'_{4a-2} & = & (\q^{4a-3},19-40a,\q,15-40a,-18+40a,11-40a,-14+40a, -10+40a,\\
&&-7+40a,-3+40a,5-40a,\q,\gr{\q},\q,\gr{2-40a},\q^{2a-7}),\\
R'_{4a} & = & (\q^{4a-1},-1-40a,\q,-5-40a,2+40a,-9-40a,6+40a, 10+40a,\\
&&13+40a,17+40a,\gr{-18-40a},\q,\gr{-15-40a},\q^{2a-7}),\\
 \end{eqnarray*}
 \end{footnotesize}
 \begin{footnotesize}
 \begin{eqnarray*}
R'_{4a+1} & = & (\q^{4a-3},-17+40a,\q,-13+40a,16-40a,-9+40a,12-40a,  6-40a,\\
&&8-40a,1-40a,-4+40a,\q,\gr{\q},\q,\gr{40a},\q^{2a-7}),\\
R'_{4a+2} & = & (\q^{4a+1},-21-40a,\q,-25-40a,22+40a,-29-40a,26+40a,  30+40a,\\
&&\gr{-35-40a},37+40a,\gr{33+40a},\q,-38-40a,\q^{2a-9}),\\
R'_{4a+3} & = & (\q^{4a-1},3+40a,\q,7+40a,-4-40a,11+40a,-8-40a, -14-40a,\\
&&-12-40a,-19-40a,\gr{20+40a},\q,\gr{16+40a},\q^{2a-7}),\\
R'_{4a+4} & = & (\q^{4a+3},-41-40a,\q,-45-40a,42+40a,-49-40a,\gr{53+40a},  50+40a,\\
&&\gr{46+40a},57+40a,-55-40a,\q,-58-40a,\q^{2a-11}),\\
R'_{4a+5} & = & (\q^{4a+1},23+40a,\q,27+40a,-24-40a,31+40a,-28-40a,  -34-40a,\\
&&\gr{36+40a},-39-40a,\gr{-32-40a},\q,40+40a,\q^{2a-9}),\\
R'_{4a+6} & = & (\q^{4a+5},-61-40a,\q,-65-40a,\gr{66+40a},-69-40a,\gr{62+40a},  70+40a,\\
&&73+40a,77+40a,-75-40a,\q,-78-40a,\q^{2a-13}),\\
R'_{4a+7} & = & (\q^{4a+3},43+40a,\q,47+40a,-44-40a,51+40a,\gr{-52-40a},  -54-40a,\\
&&\gr{-48-40a},-59-40a,56+40a,\q,60+40a,\q^{2a-11}),\\
R'_{4a+8} & = & (\q^{4a+7},-81-40a,\gr{82+40a},-85-40a,\gr{\q},-89-40a,86+40a,  90+40a,\\
&&93+40a,97+40a,-95-40a,\q,-98-40a,\q^{2a-15}),\\
R'_{4a+9} & = & (\q^{4a+5},63+40a,\q,67+40a,\gr{-68-40a},71+40a,\gr{-64-40a}, -74-40a,\\
&&-72-40a,-79-40a,76+40a,\q,80+40a,\q^{2a-13}),\\
R'_{4a+11} & = & (\q^{4a+7},83+40a,\gr{-84-40a},87+40a,\gr{\q},91+40a,-88-40a,  -94-40a,\\
&&-92-40a,-99-40a,96+40a,\q,100+40a,\q^{2a-15}).
\end{eqnarray*}

\end{footnotesize}
\noindent Notice that $\S(R_t')$ is obtained from $\S(R_t)$ replacing $A_t$ by $B_t$, where

\begin{footnotesize}
$$\begin{array}{lll}
t & A_t & B_t \\\hline
4a-2 & \emptyset & \emptyset\\
4a &  (18+40a ) &  ( 15+40a ) \\
4a+1 & \emptyset & \emptyset\\
4a+2 & (36+40a, 73+80a ) &   ( 5,  -32 -40a )\\
4a+3 & (-20-40a ) &  (  -16-40a )\\
4a+4 & ( 3, -47-40a ) &   (  10,  -40 -40a )\\
4a+5 & (-76-80a, -37-40a ) &  ( -8,  31+40a) \\
4a+6 & (-133-80a,  -64-40a ) &  ( -129 -80a, -60  -40a )\\
4a+7 & (-5, 49+40a ) &   (   -9,  45+40a )\\
4a+8 & ( -166-80a ) &   (    1 )\\
4a+9 & (  66+40a,  137+80a ) &   (  62+40a,  133+80a )\\
4a+11 & (  170+ 80a ) &    (  -1 )
\end{array}$$

\end{footnotesize}
\noindent Now, one has only to check that the elements of  $B_t$ and those of $\S(R_t)\setminus A_t$ are
pairwise distinct modulo $20n+1$ and $20n+2$. So $H'$ is an $\SH^*(n;10)$.

Suppose now $n\equiv 2 \pmod 6$, say $n=6a+2$. In this case $H$ is an  $\SH(n;10)$ such that neither its rows nor its columns are simple modulo $20n+2$. More precisely, the row $R_{4a+3}$ and the column $C_{4a+6}$ are not simple modulo $20n+2$. We construct a globally simple $\H(n;10)$ which satisfies also condition $(*)$ interchanging two rows and two
columns of $H$.

Firstly, we interchange the columns
$C_{4a+7}$ and $C_{4a+9}$ of $H$, obtaining a new array $H'$. Since
\begin{footnotesize}
\begin{eqnarray*}
C_{4a+7} & = & (\q^{4a-5},22-40a,\q,5-40a,-20+40a,13+40a,-4+40a, 26+40a,\\
&&-12-40a,42+40a,-28-40a,\q,-44-40a,\q,\q,\q^{2a-7})^T,\\
 \end{eqnarray*}
 \end{footnotesize}
 \begin{footnotesize}
 \begin{eqnarray*}
C_{4a+9} & = & (\q^{4a-5},\q,\q,2-40a,\q,-15-40a,40a,33+40a,16+40a,  46+40a,\\
&&-32-40a,62+40a,-48-40a,\q,-64-40a,\q^{2a-7})^T,
\end{eqnarray*}

\end{footnotesize}
\noindent looking at the position of the empty cells,
it is easy to see that interchanging them only the rows
$R_{(4a-5)+1}$, $R_{(4a-5)+3}$, $R_{(4a-5)+4},\ldots, R_{(4a-5)+12}$, $R_{(4a-5)+14}$ change.
We write explicitly the new rows, denoted by $R'_t$. As before the interchanged elements are in bold (for simplicity we assume $n\geq 50$, the cases $n=14,20,26,32,38,44$ can be found in \cite{web}):

\begin{footnotesize}
  \begin{eqnarray*}
R'_{4a-4} & = & (\q^{4a-5},39-40a,\q,35-40a,-38+40a,31-40a,-34+40a, -30+40a,\\
&&-27+40a,-23+40a,25-40a,\q,\gr{\q},\q,\gr{22-40a},\q^{2a-7}),\\
R'_{4a-2} & = & (\q^{4a-3},19-40a,\q,15-40a,-18+40a,11-40a,-14+40a,  -10+40a,\\
&&-7+40a,-3+40a,\gr{2-40a},\q,\gr{5-40a},\q^{2a-7}),\\
R'_{4a-1} & = & (\q^{4a-5},-37+40a,\q,-33+40a,36-40a,-29+40a,32-40a, 26-40a,\\
&&28-40a,21-40a,-24+40a,\q,\gr{\q},\q,\gr{-20+40a},\q^{2a-7}),\\
R'_{4a} & = & (\q^{4a-1},-1-40a,\q,-5-40a,2+40a,-9-40a,6+40a,  10+40a,\\
&&\gr{-15-40a},17+40a,\gr{13+40a},\q,-18-40a,\q^{2a-9}),\\
R'_{4a+1} & = & (\q^{4a-3},-17+40a,\q,-13+40a,16-40a,-9+40a,12-40a,  6-40a,\\
&&8-40a,1-40a,\gr{40a},\q,\gr{-4+40a},\q^{2a-7}),\\
R'_{4a+2} & = & (\q^{4a+1},-21-40a,\q,-25-40a,22+40a,-29-40a,\gr{33+40a},  30+40a,\\
&&\gr{26+40a},37+40a,-35-40a,\q,-38-40a,\q^{2a-11}),\\
R'_{4a+3} & = & (\q^{4a-1},3+40a,\q,7+40a,-4-40a,11+40a,-8-40a,  -14-40a,\\
&&\gr{16+40a},-19-40a,\gr{-12-40a},\q,20+40a,\q^{2a-9}),\\
R'_{4a+4} & = & (\q^{4a+3},-41-40a,\q,-45-40a,\gr{46+40a},-49-40a,\gr{42+40a},  50+40a,\\
&&53+40a,57+40a,-55-40a,\q,-58-40a,\q^{2a-13}),\\
R'_{4a+5} & = & (\q^{4a+1},23+40a,\q,27+40a,-24-40a,31+40a,\gr{-32-40a},  -34-40a,\\
&&\gr{-28-40a},-39-40a,36+40a,\q,40+40a,\q^{2a-11}),\\
R'_{4a+6} & = & (\q^{4a+5},-61-40a,\gr{62+40a},-65-40a,\gr{\q},-69-40a,66+40a,  70+40a,\\
&&73+40a,77+40a,-75-40a,\q,-78-40a,\q^{2a-15}),\\
R'_{4a+7} & = & (\q^{4a+3},43+40a,\q,47+40a,\gr{-48-40a},51+40a,\gr{-44-40a},  -54-40a,\\
&&-52-40a,-59-40a,56+40a,\q,60+40a,\q^{2a-13}),\\
R'_{4a+9} & = & (\q^{4a+5},63+40a,\gr{-64-40a},67+40a,\gr{\q},71+40a,-68-40a,  -74-40a,\\
&&-72-40a,-79-40a,76+40a,\q,80+40a,\q^{2a-15}).
\end{eqnarray*}

\end{footnotesize}
\noindent Note that $\S(R_t')$ is obtained from $\S(R_t)$ replacing $A_t$ by $B_t$, where

\begin{footnotesize}
$$\begin{array}{lll}
t & A_t & B_t \\\hline
4a-4 & \emptyset &  \emptyset \\
4a-2 & ( -2+40a ) &   ( -5+   40a ) \\
4a-1 & \emptyset &\emptyset \\
4a & (  16+40a,   33+ 80a ) & (  5,-12  -40a )\\
4a+1 & ( -40a ) & (  4 -40a\}\\
4a+2 & (  3,   -27-40a ) & (  10, -20 -40a ) \\
4a+3 & (  -36 -80a,  -17 -40a ) &  ( -8,  11+ 40a  )\\
4a+4 & (  -93  -80a,  -44 -40a ) & (  -89-80a,  -40-40a )\\
4a+5 & ( -5,    29+40a ) & ( -9,   25+40a ) \\
4a+6& ( -126  -80a ) & (  1 ) \\
4a+7 & ( 46+   40a,  97+  80a)  & ( 42+40a,  93+ 80a )\\
4a+9 & (  130+80a ) & (  -1 )
\end{array}$$

\end{footnotesize}
\noindent As before, one has only to check that the elements of  $B_t$ and those of $\S(R_t)\setminus A_t$ are pairwise distinct modulo $20n+1$ and $20n+2$.

Next, we interchange the rows $R'_{4a+3}$ and $R'_{4a+4}$ of $H'$ obtaining a new array $H''$.
It is easy to see that after this operation only the columns
$C'_{(4a-1)+1}$, $C'_{(4a-1)+3}$, $C'_{(4a-1)+4},\ldots, C'_{(4a-1)+14}$, $C'_{(4a-1)+16}$ of $H'$ have been modified. Since we are swapping two consecutive rows, in each of these columns two consecutive elements are interchanged. If one of them is an empty cell, the partial sums of this column remain the same. So, we omit to write explicitly these columns and it remains only to consider the following ones:

\begin{footnotesize}
\begin{eqnarray*}
C''_{4a+4} & = & (\q^{4a-5},-23+40a,\q,-10+40a,21-40a,-9-40a,6-40a,  -25-40a,\\
&&\gr{-41-40a},\gr{11+40a},27+40a,\q,43+40a,\q^{2a-5})^T,\\
C''_{4a+6} & = & (\q^{4a-3},-3+40a,\q,10+40a,1-40a,-29-40a,\gr{-45-40a},  \gr{-14-40a},\\
&&31+40a,-61-40a,47+40a,\q,63+40a,\q^{2a-7})^T,\\
C''_{4a+7} & = & (\q^{4a-3},2-40a,\q,-15-40a,40a,33+40a,\gr{46+40a},  \gr{16+40a},\\
&&-32-40a,62+40a,-48-40a,\q,-64-40a,\q^{2a-7})^T,\\
C''_{4a+8} & = & (\q^{4a-1},17+40a,\q,30+40a,\gr{-49-40a},\gr{-19-40a},-34-40a,  -65-40a,\\
&&51+40a,-81-40a,67+40a,\q,83+40a,\q^{2a-9})^T,\\
C''_{4a+9} & = & (\q^{4a-5},22-40a,\q,5-40a,-20+40a,13+40a,-4+40a,  26+40a,\\
&&\gr{42+40a},\gr{-12-40a},-28-40a,\q,-44-40a,\q^{2a-5})^T,\\
C''_{4a+11} & = & (\q^{4a-1},-18-40a,\q,-35-40a,\gr{53+40a},\gr{20+40a},36+40a,  66+40a,\\
&&-52-40a,82+40a,-68-40a,\q,-84-40a,\q^{2a-9})^T.\\
\end{eqnarray*}

\end{footnotesize}
\noindent
Notice that $\S(C_t'')$ is obtained from $\S(C_t')$ replacing $A_t$ by $B_t$, where
\begin{footnotesize}
$$\begin{array}{lll|lll}
t & A_t & B_t & t & A_t & B_t\\\hline
4a+4& ( -29-40a ) &  ( -81 -120a ) &
4a+8 & ( 28+40a ) & ( -2+40a )\\
4a+6 & ( -35-40a ) & ( -66-40a ) &
4a+9 & (30+40a ) & ( 84+120a) \\
4a+7 & (36+ 40a ) & ( 66+ 40a) &
4a+11 & ( -33-40a) &  (-40a)
\end{array}$$

\end{footnotesize}
\noindent Again by a direct check, one can verify that the partial sums of these $6$ columns are pairwise distinct both modulo $20n+1$ and $20n+2$. Hence, $H''$ is an $\SH^*(n;10)$ for any $n\equiv 2 \pmod 4$. This concludes the proof.
\end{proof}

\begin{ex}
Let $n=12$. By the construction given in the proof of Proposition \ref{prop:10}, we obtain the following
$\SH^*(12;10)$:
\begin{footnotesize}
$$
\begin{array}{|c|c|c|c|c|c|c|c|c|c|c|c|}\hline
-104 & 111 & -108 & -114 & -112 & -119 & 116 &  & 120 & 103 &  & 107 \\\hline
-38 & -21 &  & -25 & 22 & -29 & 26 & 30 & 33 & 37 & -35 &  \\\hline
 & 7 & -4 & 11 & -8 & -14 & -12 & -19 & 16 &  & 20 & 3 \\\hline
-55 &  & -58 & -41 &  & -45 & 42 & -49 & 46 & 50 & 53 & 57 \\\hline
40 & 23 &  & 27 & -24 & 31 & -28 & -34 & -32 & -39 & 36 &  \\\hline
73 & 77 & -75 &  & -78 & -61 &  & -65 & 62 & -69 & 66 & 70 \\\hline
56 &  & 60 & 43 &  & 47 & -44 & 51 & -48 & -54 & -52 & -59 \\\hline
86 & 90 & 93 & 97 & -95 &  & -98 & -81 &  & -85 & 82 & -89 \\\hline
-72 & -79 & 76 &  & 80 & 63 &  & 67 & -64 & 71 & -68 & -74 \\\hline
102 & -109 & 106 & 110 & 113 & 117 & -115 &  & -118 & -101 &  & -105 \\\hline
-88 & -94 & -92 & -99 & 96 &  & 100 & 83 &  & 87 & -84 & 91 \\\hline
 & -5 & 2 & -9 & 6 & 10 & 13 & 17 & -15 &  & -18 & -1 \\\hline
\end{array}
$$
\end{footnotesize}
\end{ex}

\section{Conclusions}

\begin{thm}\label{theo:SH}
Let $3\leq k\leq 10$. Then there exists an $\SH^*(n;k)$ if and only if $n\geq k$ and  $nk \equiv 0,3 \pmod 4$.
\end{thm}

\begin{proof}
If $k=3,4,5$ the results follow from Theorem \ref{th:integer} and from the considerations on partial sums given at the beginning of Section \ref{sec4}.
So, assume $6\leq k\leq 10$. We recall that, by Theorem \ref{th:integer}, an $\H(n;k)$ exists only when $n\geq k$ and
$nk\equiv 0,3 \pmod 4$. For these cases, we give in the following table the proposition number where we constructed an
$\SH^*(n;k)$ (the first column $n_4$ gives the congruence class of $n$ modulo $4$).

\begin{center}
\begin{tabular}{|c|ccccc|}\hline
{$n_4$}$\diagdown${$k$} & $6$ & $7$ & $8$ & $9$ & $10$ \\ \hline
0 & \ref{prop:6} & \ref{SH7_even} & \ref{0,2} & \ref{SH9_even} & \ref{prop:10} \\
1 & $\nexists$ & \ref{SH7_odd} & \ref{1,3} & $\nexists$ &  $\nexists$ \\
2 & \ref{prop:6} & $\nexists$& \ref{0,2} & $\nexists$ & \ref{prop:10} \\
3 & $\nexists$ & $\nexists$ & \ref{1,3} & \ref{SH9_odd} & $\nexists$ \\\hline
\end{tabular}
\end{center}
\end{proof}

Now, Theorem \ref{th:main} easily follows from Theorem \ref{theo:SH} applying Proposition \ref{pr:ccp}.
The cases $1\leq n< k$ have been obtained with the help of a computer starting from the constructions given in \cite{BGL,BDF}, see \cite{web}.

\begin{proof}[Proof of Theorem \ref{thm:biemb}]
The $\SH^*(n;7)$ for $n\equiv 1 \pmod 4$ and the $\SH^*(n;9)$ for $11<n\equiv 3 \pmod 4$ obtained in Section \ref{sec4}
are cyclically $7$-diagonal and cyclically $9$-diagonal, respectively. The $\H(n;3)$ for $n\equiv 1 \pmod 4$ of \cite{ADDY} is cyclically $3$-diagonal.
Let $A=(a_{i,j})$ be  the $\H(n;5)$ described in \cite{DW} for $n\equiv 3 \pmod 4$ and define $h_{2i,2j}=a_{i,j}$.
Then $H=(h_{i,j})$ is a
cyclically $5$-diagonal $\SH^*(n;5)$.
By Propositions \ref{biembk} and \ref{biemb7} in each of these cases there are simple compatible orderings of the rows
and columns. Then, the result follows from Theorem \ref{Gustin}.

For the exceptional $\SH^*(11;9)$ described in \cite{web}, take as $\omega_r$ the natural
ordering of the rows from left to right and as $\omega_c$ the natural ordering of the columns from top to bottom for
the first $9$ columns, from bottom to top for the last two columns. Then $\omega_r,\omega_c$ are compatible orderings.
Again, we apply  Theorem \ref{Gustin}.
\end{proof}

\section*{Acknowledgements}
The authors would like to thank the anonymous referees for their useful suggestions and comments.

\end{document}